\documentclass[leqno]{amsart}

\newcommand{\old}[1]{}

\renewcommand{\emph}[1]{\textit{#1}}
\usepackage{enumerate,amsmath,latexsym,amssymb,amsthm}
\usepackage{pdftricks}
\begin{psinputs}
  \usepackage{pstricks}
\end{psinputs}
\usepackage{color}\usepackage{graphicx}

\definecolor{brown}{cmyk}{0, 0.72, 1, 0.45}
\definecolor{grey}{gray}{0.5}

 \allowdisplaybreaks

\newcounter{rot}

\newcommand{\empt}{\varnothing}

\def\wLB{\widehat{\lb}}
\newcommand{\ignore}[1]{}

\def\cX{{\mathcal X}}
\def\cV{{\mathcal V}}
\def\cY{{\mathcal Y}}
\def\cI{{\mathcal I}}

\newcommand{\set}[1]{\left\{#1\right\}}

\def\ii_(#1,#2){i_{#1}^{#2}}

\renewcommand{\S}[1]{S^{(#1)}}
\newcommand{\bS}[1]{{\bar S}^{(#1)}}

\def\a{\alpha}
\def\b{\beta}
\def\d{\delta}
\def\D{\Delta}
\def\e{\varepsilon}

\def\g{\gamma}

\def\th{\theta}

\def\n{\nu}
\def\p{\pi}

\def\r{\rho}

\def\t{\tau}

\def\Om{\Omega}

\def\x{\xi}

\newcommand{\sbs}{\subseteq}
\newcommand{\sbsn}{\subsetneq}
\newcommand{\stm}{\setminus}

\def\cN{{\mathcal N}}

\def\cS{{\mathcal S}}

\def\Re{\mathbb{R}}

\newcommand{\brac}[1]{\left( #1 \right)}

\def\hx{\hat{x}}

\newcommand{\expect}{\operatorname{\bf E}}
\def\E{\expect}
\renewcommand{\Pr}{\operatorname{\bf Pr}}
\newcommand\bfrac[2]{\left(\frac{#1}{#2}\right)}

\newtheorem{theorem}{Theorem}[section]
\newtheorem*{theoremRinner}{Theorem~\ref{\repthmref} (restated)}
\newenvironment{theoremR}[1]
{\def\repthmref{#1}\theoremRinner}{\endtheoremRinner}

\newtheorem{conjecture}[theorem]{Conjecture}
\newtheorem{lemma}[theorem]{Lemma}
\newtheorem{proposition}[theorem]{Proposition}

{
\theoremstyle{definition}
\newtheorem{remark}[theorem]{Remark}
\newtheorem{q}{}
}

\newtheorem{observation}[theorem]{Observation}
\newcounter{thmtemp}

\newcommand{\nospace}[1]{}

\def\path{\operatorname{PATH}}

\def\V{{\bf Var}}
\newcommand{\beq}[1]{\begin{equation}\label{#1}}
\def\eeq{\end{equation}}

\renewcommand{\Re}{\mathbb{R}}

\def\cL{{\mathcal L}}

\def\La{\Lambda}

\newcommand{\de}[2]{||#1-#2||}

\newcommand{\diam}{\mathrm{diam}}
\newcommand{\dist}{\mathrm{dist}}

\newcommand\xdn{\cX_n}
\newcommand\ydn{\cY_n}

\newcommand{\flr}[1]{\lfloor #1 \rfloor}

\begin{document}

\newcommand{\tf}{{\mathrm{TF}}}
\newcommand{\hk}{{\mathrm{HK}}}
\newcommand{\hf}{{H\mathrm{F}}}
\newcommand{\etf}{{\mathrm{E}2\mathrm{F}}}
\newcommand{\mst}{\mathrm{MST}}
\newcommand{\tsp}{\mathrm{TSP}}
\newcommand{\lb}{\mathrm{LB}}
\newcommand{\mm}{\mathrm{MM}}
\newcommand{\tmm}{\mathrm{2MM}}
\newcommand{\bmst}{\b_\mst}
\newcommand{\bmsT}[1]{\b_{\mst_{#1}}}
\newcommand{\bmstk}{\b_{\mst_k}}
\newcommand{\btsp}{\b_\tsp}
\newcommand{\bmm}{\b_\mm}
\newcommand{\btf}{\b_\tf}
\newcommand{\blb}{\b_\lb}
\newcommand{\bhk}{\b_\hk}
\newcommand{\bhf}{\b_\hf}
\newcommand{\betf}{\b_\etf}
\newcommand{\btfg}{\b_{\tf_g}}
\newcommand{\btF}[1]{\b_{\tf_{#1}}}
\newcommand{\betfg}{\b_{\etf_g}}
\newcommand{\btfG}[1]{\b_{\tf_{#1}}}
\newcommand{\betfG}[1]{\b_{\etf_{#1}}}

\title{Separating subadditive Euclidean functionals}

\author{Alan Frieze}
\address{Department of Mathematical Sciences\\
Carnegie Mellon University\\
Pittsburgh, PA 15213\\
U.S.A.}
\email[Alan Frieze]{alan@random.math.cmu.edu}
\thanks{Research supported in part by NSF grant DMS-1362785.}
\author{Wesley Pegden}
\email[Wesley Pegden]{wes@math.cmu.edu}
\thanks{Research supported in part by NSF grant DMS-1363136.}

\date{\today}

\begin{abstract}
If we are given $n$ random points in the hypercube $[0,1]^d$, then the minimum length of a Traveling Salesperson Tour through the points, the minimum length of a spanning tree, and the minimum length of a matching, etc., are known to be asymptotically $\beta n^{\frac{d-1}{d}}$ a.s., where $\beta$ is an absolute constant in each case.  We prove separation results for these constants.  In particular, concerning the constants $\btsp^d$, $\bmst^d$, $\bmm^d$, and $\btf^d$ from the asymptotic formulas for the minimum length TSP, spanning tree, matching, and 2-factor, respectively,  we prove that $\bmst^d<\btsp^d$, $2\bmm^d<\btsp^d$, and $\btf^d<\btsp^d$ for all $d\geq 2$.  We also asymptotically separate the TSP from its linear programming relaxation in this setting. Our results have some computational relevance, showing that a certain natural class of simple algorithms cannot solve the random Euclidean TSP efficiently.
\end{abstract}

\maketitle
\section{Introduction}
Beardwood, Halton, and Hammersley \cite{BHH} studied the length of a Traveling Salesperson Tour through random points in Euclidean space.  In particular, if $x_1,x_2,\dots$ is a random sequence of points in $[0,1]^d$ and $\xdn=\{x_1,\dots,x_n\}$, their results imply that there is an absolute constant $\btsp^d$ such that the length $\tsp(\xdn)$ of a minimum length tour through $\xdn$ satisfies
\begin{equation}
\label{e.bhh}
\tsp(\xdn)\sim \btsp^d n^{\frac{d-1}{d}}\qquad a.s.
\end{equation}

This result has many extensions; for example, we know that identical asymptotic formulas hold for the the cases of the minimum length of a spanning tree $\mst(\xdn)$\cite{BHH}, and the minimum length of a matching $\mm(\xdn)$ \cite{Papa}.  Steele \cite{S} provided a general framework which enables fast assertion of identical asymptotic formulas for these and other suitable problems.  For example, we will see in Section \ref{s.EF} that his results imply that the length $\tf(\xdn)$ of a minimum length 2-factor admits the same asymptotic characterization.

A major remaining problem in this area is to obtain analytic results regarding the constants $\b$ in such formulas.  In particular, the best rigorous bounds on such constants are generally very weak, with known results for $d=2$ given in Table \ref{t.betas}.  In particular, the bounds on $\btsp^2$ were not improved since 1959, until the paper \cite{Steiner} of Steinerberger improved the lower bound by $\tfrac{19}{5184}\approx .0036\dots$ and the upper bound by $\approx 10^{-6}$ .  On the other hand, there was some success as $d$ grows large, as Bertsimas and Van Ryzen \cite{ind} showed that, asymptotically in $d$, 
\begin{equation}
\bmst^d\sim 2\bmm^d \sim \sqrt{\frac{d}{2\pi e}},
\end{equation}
and conjectured that $\btsp^d\sim \sqrt{\frac{d}{2\pi e}}$ as well.
\begin{table}
\renewcommand{\arraystretch}{1.2} 
\begin{tabular}{c|ll}
 & lower & upper\\
\hline
\hline
$\btsp^2$ & .62866 \cite{Steiner}& .92116\dots \cite{BHH,Steiner}\\
\hline
$\bmst^2$ & .60082 \cite{MSTseries} & $\frac{1}{\sqrt 2}\approx.707$  \cite{Gil} \\
\hline
$2\bmm^2$ & $.5$ \cite{ind} & $.92116$ 
\end{tabular}
\vspace{1ex}
\caption{\label{t.betas} Bounds on constants for $d=2$.}
\end{table}

\bigskip
It seems that it has been overlooked that local geometric arguments are sufficient to prove the separation of constants for many natural examples of Euclidean functionals.   In particular, in the present paper, we will show that $\bmst^d<\btsp^d$, $\btf^d<\btsp^d$,  and $2\bmm^d<\btsp^d$ for all $d$.  These are the first asymptotic separations for Euclidean functionals where the Euclidean metric is playing an essential role: the only previous separation was shown (by Bern \cite{Bern}; see also \cite{HY}) for the minimum length rectilinear Steiner tree vs. the minimum rectilinear length spanning tree, which is equivalent to asymptotically distinguishing Steiner trees from trees in the $L_1$ norm.   (The rectilinear Steiner tree is also the only case where the asymptotic \emph{worst-case} length is known exactly \cite{CG}.)  Finally, we will also asymptotically separate the TSP from its linear programming relaxation.  

We begin by considering the degrees of vertices in the minimum spanning trees among $n$ random points.   Steele, Shepp, and Eddy \cite{leaves} showed that the number $\Lambda_k(\xdn)$ of vertices of degree $k$ satisfies
\[
\Lambda_k(\xdn)\sim \alpha_{k,d} n
\]
for constants $\alpha_{k,d}$, and proved that $\alpha(1,d)>0$.  Note that we must have $\alpha_{k,d}=0$ when $k>\tau(d)$, where $\tau(d)$ is the kissing number of $d$ dimensional space (6 in the case $d=2$).  Indeed, we must have $\a_{k,d}=0$ whenever $k>\tau'(d)$, where $\tau'(d)$ denotes a \emph{strict kissing number} of $d$, which we define as the maximum $K$ such that there exists $\e>0$ such that there is, in $d$ dimensions, a configuration of $K$ disjoint spheres of radius $1+\e$ each tangent to a common unit sphere.  (Note that $\tau'(d)\leq \tau(d)$, and in particular, $\tau'(2)=5$.)  We prove:
\begin{theorem}\label{t.posdeg}
$\alpha(k,d)>0$ if $k\leq \tau'(d)$.
\end{theorem}

Considering Euclidean functionals $\mst_k(\cX)$ (with corresponding constants $\bmstk^d$) defined as the minimum length of a spanning tree of $\cX$ whose vertices all have degree $\leq k$, we will then get separation as follows:
\begin{theorem}\label{t.mstseq}
We have that 
\begin{equation}
\btsp^d=\bmsT2^d>\bmsT3^d>\cdots >\bmsT{\t'(d)}^d=\bmst^d 
\end{equation}
for all $d$.
\end{theorem}

Thus, the $\mst_k$ constants are as diverse as are allowed by the simple geometric constraint of $\tau'(d)$.

Still, there are only finitely many constants $\bmsT k^d$ for each $d$; while we can draw trees with very large degrees, large degrees (relative to $d$) are not useful for minimum spanning trees in Euclidean space.   In contrast to this scenario,  let us recall that a \emph{2-factor} is a disjoint set of cycles covering a given set of points.  We will see in Section \ref{s.EF} that the length of the minimum 2-factor is indeed a subadditive Euclidean functional, and thus this length satisfies $\tf(\xdn)\sim \btf^d n^{\frac{d-1}{d}}$ for some constant $\btf^d$.  Moreover, if $\tf_g(X)$ is the minimum length of a 2-factor through $X$ whose cycles all have length $\geq g$, then we will see that $\tf_g$ is also a subadditive linear functional, so that we have $\tf_g(\xdn)\sim \btfg^d n^{\frac{d-1}{d}}.$  Naturally, we must have $\btf^d=\btfG3^d\leq \btfG4^d\leq \btfG5^d\leq \cdots.$  In analogy to the high-degree vertices in a tree, we can of course draw 2-factors with small cycles, but it is not clear \emph{a priori} whether small cycles will be asymptotically essential to optimum 2-factors in random point sets.  The following theorem shows that they are:

\begin{theorem}\label{t.girthseq}
$\btfg$ is a monotone increasing sequence $\btfG3^d<\btfG4^d<\btfG5^d<\cdots $.  
\end{theorem}

On the other hand, we prove that 2-factors with long (but constant) girth requirements produce close approximations to the TSP:
\begin{theorem}\label{t.2flimit}
 $\lim\limits_{g\to \infty} \btfg^d=\btsp^d$.
\end{theorem}
With a bit more work, our method for proving Theorem \ref{t.girthseq} will also allow us to deduce the following:
\begin{theorem}\label{t.mmnottsp}
$2\bmm^d<\btsp^d$.
\end{theorem}
We note in contrast to Theorem \ref{t.girthseq} that in the independent case where the edge lengths $X_e,e\in \binom{[n]}{2}$ are independent uniform $[0,1]$ random variables, Frieze \cite{F00} showed that with probability $1-o(1)$, the weight of the minimal 2-factor is asymptotically equal to the minimum length of a tour.

We continue by mentioning a natural generalization of $\mm(\xdn)$.  Given a fixed graph $H$ on $k$ vertices, an $H$-factor of a set of points $S$ is a set of edges isomorphic to $\flr{|X|/k}$ vertex disjoint copies of $H$.  As a subadditive Euclidean functional, the minimum length $\hf(\xdn)$ of an $H$ factor of $\xdn$ satisfies
\[
\hf(\xdn)\sim \b_H^d n^{\frac{d-1}{d}}.
\]
We pose the following conjecture:
\begin{conjecture}\label{ccc}
Given $H_1,H_2$ and $d\geq 2$, we have that $\b^d_{H_1}\neq \b^d_{H_2}$ unless $H_1$ and $H_2$ are each isomorphic to a disjoint union of copies of some graph $H_3$.  In particular, $\b^d_{H_1}\neq \b^d_{H_2}$ if $H_1,H_2$ are connected and non-isomorphic.
\end{conjecture}
We prove at least the following, showing diversity in the constants even for fixed edge density:
\begin{theorem}\label{t.hfac}
For any fixed $d\geq 2$ and rational $r\geq 1$, there are infinitely many distinct constants $\b^d_H$ over connected graphs $H$ with edge density $\frac{|E(G)|}{|V(G)|}=r$.
\end{theorem}

Our final separation result concerns the linear programming relaxation of the TSP.  The TSP through a set of points can be given as the following integer program, on variables $x_{\{ij\}}$ $(i,j\in V, i<j)$ indicating the presence of an edge between vertices $i$ and $j$ in the tour, where $c_{\{ij\}}$ gives the cost of the edge $\{i,j\}$:
\newcounter{eq}
\setcounter{eq}{0}

\[
\arraycolsep=1.4pt\def\arraystretch{2.2}
\begin{array}{rrl}
\multicolumn{3}{c}{\displaystyle    \min \sum_{\{i,j\}\sbs V} c_{\{ij\}}x_{\{ij\}}}\\
\multicolumn{3}{c}{\mbox{ subject to }}\\
\hline
    (\forall i)& \displaystyle \sum_{j\neq i} x_{\{ij\}}&=2\\
  (\forall \empt\neq S\subsetneq V)& \displaystyle\sum_{\{i,j\}\sbs S}x_{\{ij\}}&\leq |S|-1\\
   (\forall i<j\in V)&\displaystyle x_{\{ij\}}&\in \{0,1\}\\
\end{array}
\]

The linear programming relaxation of the above integer program replaces the final constraint with the requirements that $0\leq x_{\{ij\}}\leq 1$ for each $x_{\{ij\}}$. It is often referred to as the Held-Karp relaxation. of the TSP, but its origins go back to the the paper of Dantzig, Fulkerson and Johnson \cite{DFJ54}.  It should be noted that, although this LP has an exponential number of constraints, it can be solved in polynomial time, e.g., using the ellipsoid algorithm \cite{GLS}.

We denote by $\hk(\xdn)$ the value of a solution to the Held-Karp relaxation on $\xdn$.  Of course, $\hk(\xdn)\leq \tsp(\xdn)$.  The Held-Karp bound on the TSP is generally considered to be a good bound which is algorithmically useful on `typical instances' (see \cite{C,HeK,J,ST,VJ}), and, with this motivation, Goemans and Bertsimas \cite{GB} showed that $\hk(\xdn)\sim \bhk n^{\frac{d-1}{d}}$ for some constant $\bhk$, and asked whether $\bhk=\btsp$.  Experimental evidence has suggested that any gap between the constants would be less than $1\%$. \cite{JMR}

\begin{theorem}\label{t.hk}For all $d\geq 2$,
$\bhk^d<\btsp^d$.
\end{theorem}

Our separation results have implications for the practical problem of solving the Euclidean TSP.  \emph{Branch and bound} algorithms are a standard approach to solving NP-hard problems, in which a bounding estimate is used to prune an exhaustive search of the solution space.  
There has been a great deal of success solving real-world instances of the TSP with branch-and-bound augmented with sophisticated techniques based on cutting planes for the TSP polytope (see, for example Applegate, Bixby, Chv\'atal and Cook \cite{ABCC}). 

Our results show, however, show that several natural approximators (2-factor, twice a matching, the Held-Karp bound, etc) are asymptotically distinct from the TSP.  This will be algorithmically relevant as follows: 
\begin{theorem}\label{bbth}
Suppose that we use branch and bound to solve the TSP on $\xdn$, using $\tf_g$ or $\hk$ as a lower bound. Then, w.h.p, the algorithm runs in time $e^{\Omega(n/\log^6 n)}$.
\end{theorem} 
In particular, this gives a rigorous explanation for the observation (see \cite{MP}, for example) that branch-and-bound heuristics using the Assignment Problem as a bounding estimate (even weaker than the 2-factor) perform poorly on random Euclidean instances, and indicates that the success of the Held-Karp bound in branch and bound algorithms will be limited for sufficiently large Euclidean point sets.

\begin{remark}
Asymptotic formulas are available for subadditive Euclidean functionals in more general settings.  If $x_1,x_2,\dots \in \Re^d$ are independent identically distributed random variables with bounded support, then the length $L(x_1,\dots,x_n)$ of the functional on the points $\xdn$ satisfies
\[
L(x_1,\dots,x_n)/n^{\frac{d-1}{d}}\to \b_L^d \int_{\Re^d}f(x)^{\frac{d-1}{d}}dx,
\]
where $f$ is the absolutely continuous part of the distribution of the $x_i$'s and $\b_L^d$ is a constant depending only on $d$ and $L$ (see \cite{BHH,S}).  Note that this gives an asymptotic formula for $L(x_1,\dots,x_n)$ unless the right hand side is zero.  The latter case will happen if the $x_i$'s lie exclusively on some $m$-dimensional manifold embedded in $\Re^d$ where $m<d$, but the BHH theorem also has a suitable extension to this setting \cite{manifolds}, allowing asymptotic formulas involving $n^{\frac{m-1}{m}}$.   Our results are all immediately valid in these more general settings, however: as the constants $\b_L^d$ depend only on $L$ and $d$ or $m$ (in particular, not on the distribution or, in the second case, the particular manifold), it is enough study the constants in the case of points which are uniformly distributed in the hypercube.
\end{remark}

\section{Subadditive Euclidean Functionals}
\label{s.EF}
Steele defined a \emph{Euclidean functional} as a real valued function $L$ on finite subsets of $\Re^d$  which is invariant under translation, and scales as $L(\a X)=\a L(X)$.  It is \emph{nearly monotone} with respect to addition of points if
\begin{equation}
L(X\cup Y)\geq L(X)-o(n^{\frac{d-1}{d}})\quad\mbox{for }n=|X|.
\end{equation}
It has finite variance if, fixing $n$, we have
\begin{equation}
\V(L(\xdn))<\infty
\end{equation}
(in particular, if it is bounded for fixed $n$) and it is \emph{subadditive} if, 
for $\ydn$ a random set of $n$ points from $[0,t]^d$, it satisfies 
\[
L(\ydn)\leq \sum_{\a\in [m]^d}L(S_\a\cap \ydn)+Ctm^{d-1}
\]
for some absolute constant $C$, where here $\{S_\a\}$ ($\a\in [m]^d$) is a decomposition of $[0,t]^d$ into $m^d$ subcubes of side length $u=t/m$.

Steele proved:
\begin{theorem}[Steele \cite{S}]
If $L$ is a subadditive Euclidean functional on $\Re^d$ of finite variance, $x_1,x_2,\dots$ is a random sequence of points from $[0,1]^d$, and $\xdn=\{x_1,x_2,\dots,x_n\}$, then there is an absolute constant $\b^d_L$ s.t.
\begin{equation}\label{ratioasym}
L(\xdn)/n^{\frac{d-1}{d}}\to \b^d_L\qquad a.s.
\end{equation}
In particular, we have a.s. that either $L(\xdn)=o(n^{\frac{d-1}{d}})$ (if $\b_L^d=0$) or that
\begin{equation}\label{asym}
L(\xdn)\sim \b^d_L n^{\frac{d-1}{d}}.
\end{equation}
\end{theorem}

This can thus be used to easily give the existence of the simple asymptotic formulas for the functionals $\tf_g(X),$ $\mst_k(X)$, and $\hf(X)$ by showing that these functionals are subadditive.  We note that the constant $\b_L^d$ cannot be zero in any of these cases and thus that \eqref{ratioasym} implies \eqref{asym}, since for any $1\leq i\leq n$,
\[\E(\min\limits_{\substack{x\in \xdn\\x\neq x_i}}\dist(x_i,x))=\Omega(n^{-\frac 1 d}),
\]
 giving a lower bound of $\Omega(n^{\frac{d-1}{d}})$ for each functional by Linearity of Expectation.

\begin{proposition}
$\tf_g(X)$, $\mst_k(X)$, and $\hf(X)$ are subadditive Euclidean functionals.
\end{proposition}

Before writing a proof, we note that for the definition of the 2-factor functionals $\tf_g(X)$, we can only require that the 2-factors whose length we minimize cover all the points when there are at least $\max(g,3)$ points.  
Similarly, the $\hf(X)$ functional is required just to cover at least $n-|H|+1$ points.

\begin{proof}
We begin by noting that for each of these functionals, we can assert an upper bound $Cn^{\frac{d-1}{d}}$ for some constant $C$, even over worst-case arrangements of $n$ points in $[0,1]^{d}$.  The analogous statement for the TSP was proved by Toth \cite{Toth} and by Few \cite{Few}, and implies these bounds for the functionals considered here.  Indeed, a tour through $n$ points itself gives a tree of max-degree 2 (after deleting one edge), and is a 2-factor subject to any constant girth restriction. 
For $H$ factors, a tour can be divided into paths of length $|H|$ (except for $<|H|$ remaining vertices) which can then be completed to instances of $H'\supseteq H$ by adding edges. Each added edge has a cost bounded by the length of the path it lies in and so this construction increases the total cost by at most a factor equal to the number of edges in $H$.

Subadditivity of $\tf_g(X)$, and $\hf(X)$ is now a consequence of the fact that a union of 2-factors (subject to restrictions on the cycle length, perhaps) or $H$-factors is again a 2-factor (subject to the same restrictions) or an $H$ factor, respectively.  In particular, the subadditive error term for these functions comes just from the fact that points may be uncovered in some of the subcubes $S_\a$, for the exceptional reasons noted above.  Since there are at most $(g-1)m^d$ or $|H|m^d$ such uncovered points, however, the error is suitably bounded by the minimum cost factor on a worst-case arrangement of the remaining points.

Subadditivity of $\mst_k(X)$ ($k\geq 2$) is similar: after finding minimal spanning trees of max-degree $k$ in each subcube $S_\a$, we must join together these trees into a single tree.  We choose 2 leaves of each subcube's tree and denote one red and the other blue.  We let $\a_1,\a_2,\dots,\a_{m^d}$ denote a path through the decomposition $\{S_\a\}$, so that the subcubes $S_{\a_i}$ and $S_{\a_{i+1}}$ are adjacent.  For each $i<m^d$, we join the red leaf of the tree in $S_{\a_i}$ to the blue leaf in the tree of $S_{\a_i+1}$.  The result is a spanning tree of the whole set of points with the same maximum degree and with extra cost at most $2\sqrt d u m^{d}=2\sqrt d t m^{d-1}$.
\end{proof}
\section{Separating asymptotic constants}
In the following we will use the simplest application of the Azuma-Hoeffding martingale tail inequality: It is often referred to as McDiarmid's inequality \cite{McD}. Suppose that we have a random variable $Z=Z(X_1,X_2,\ldots,X_N)$ where $X_1,X_2,\ldots,X_N$ are independent. Further, suppose that changing one $X_i$ can only change $Z$ by at most $c$ in absolute value. Then for any $t>0$,
\beq{mcd}
\Pr(|Z-\E Z|\geq t)\leq 2\exp\set{-\frac{t^2}{c^2N}}.
\eeq
We will also use the following inequality, applicable under the same conditions, when $\E Z$ is not large enough.
\beq{mcd1}
\Pr(Z\geq \a\E Z)\leq \bfrac{e}{\a}^{\a\E Z/c}.
\eeq
Our method to distinguish constants is based on achieving constant factor improvements to the values of functions via local changes.  Given $\e,D\in \Re$ and a finite set of points $S\subseteq \Re^d$ and a universe $X$, we say that $T\subseteq X$ is an $(\e,D)$-copy of $S$ if there is a bijection $f$ between $T$ and a point set $S'$ congruent to $S$ such that $\de {x}{f(x)}<\e$ for all $x\in T$, and such that $T$ is at distance $>D$ from $X\stm T$. Here we will further assume that $||x-y||>\e$ for $x\neq y\in S$.

For our purposes, it will be convenient notationally to work with $n$ random points $\ydn$ from $[0,t]^d$ where $t=n^{1/d}$, in place of $n$ random points $\xdn$ from $[0,1]^d$. At the end, we will scale our results by a factor $n^{-1/d}$ in order to get what is claimed above.
\begin{observation}\label{o.findit}
Given any finite point set $S$, any $\e>0$, and any $D$, $\ydn$ a.s contains at least $C^S_{\e,D}n$ $(\e,D)$-copies of $S$, for some constant $C^S_{\e,D}>0$.
\end{observation}
\begin{proof}[Proof of Observation \ref{o.findit}]
Let $Z$ denote the number of $(\e,D)$-copies of $S$ in $\ydn$. We divide $[0,t]^d$ into $n/(3D)^d$ subcubes $C_1,C_2,\ldots,$ of side $3D$. Then let $C_i'\subseteq C_i$ be a centrally placed subcube of side $D$. Now choose a set $S'$ congruent to $S$ somewhere inside $C_1'$ and let $B_1,B_2,\ldots,B_s,\,s=|S|$ be the collection of balls of radius $\e$, centered at each point of $S'$. The with probability at least $\a=\a_{\e,D}>0$, each $B_i$ contains exactly one point of $\ydn$ and there are no other points of $\ydn$ in $C_1$. Thus $\E Z\geq \b n$ where $\b=\a/(3D)^d$. Now changing the position of one point in $\ydn$ changes the number of $(\e,D)$-copies of $S$ by at most two and so we can use McDiarmid's inequality \cite{McD} to show that $Z\geq \frac12\E Z$ a.s. 
\end{proof}

To use this to prove Theorem \ref{t.posdeg}, we will need just a bit  more.
\begin{observation}\label{o.convshield}
If $Y\sbs \Re^d$ and $x$ lies in the interior of the convex hull of $Y$, then when $D$ is sufficiently large, any point at distance $>D$ is closer to some point of $Y$ than to $x$.\qed
\end{observation}
 If $v_0,v_1,\dots,v_k$ are vectors in $\Re^d$ with pairwise negative dot-product, then $v_1,\dots,v_k$ lie in the half-space $v_0\cdot x<0$, and the projections of $v_1,\dots,v_k$ onto the hyperplane $v_0\cdot x=0$ have pairwise negative dot-products.  This gives the following, by induction on $d$:
\begin{observation}\label{o.dotproduct2}
If $v_1,\dots,v_{d+1}\in \Re^d$ are vectors with negative pairwise dot-products, then $0$ is a positive linear combination of the $v_i$'s.\qed 
\end{observation}
This allows us to prove:
\begin{lemma}\label{l.config}
If $d+1\leq k\leq \tau'(d)$, then there exists a set of points $\bS{k}\sbs \Re^d$ consisting of a single point at the origin, surrounded by a set $\S{k}$ of $k$ points on the unit sphere centered at the origin and separated pairwise by at least some $\e'>0$ more than unit distance, such that $\S{k}$ does not lie in open half-space whose boundary passes through the origin.
\end{lemma}
\begin{proof}
We first observe that the definition of $\tau'$ already gives us a set $\S k$ with the desired properties, except that it may all lie in some open half-space through the origin.  In this case, however, we can delete a point and replace it with the point $x_H$ on the unit sphere opposite the half-space $H$, and furthest away from the halfspace. We do this repeatedly and note that because the above exchange of points only happens when all points are on one side of a half-space $H'$, $x_H$ remains as the unique point which is in the open half-space opposite to $H$.  Furthermore, doing this repeatedly, we can achieve either a set $\S k$ with all the desired properties, or can find after at most $k$ steps a set $\S k$ of points on the sphere separated pairwise by at least $\e'>0$ more than unit distance, and whose pairwise dot products as vectors in $\Re^d$ are all negative. But then Observation \ref{o.dotproduct2} and $k\geq d+1$ implies that the points cannot all lie in the interior of some half-space whose boundary passes through the origin.
\end{proof}
We are now ready to prove Theorem \ref{t.posdeg}.
\subsection*{Proof of Theorem \ref{t.posdeg}}
Given $k\geq 2$, we choose any $d'\leq d$ such that $d'+1\leq k\leq \tau'(d')$.

We apply Lemma \ref{l.config} with $k,d'$ to get a set ${S'}^{(k)}\sbs \Re^d$.
Observe first that the origin must lie in the convex hull $X$ of the set ${S'}^{(k)}$ given by Lemma \ref{l.config}; otherwise, there would be a supporting half-space $H$ of $X$ not containing the origin, and ${S'}^{(k)}$ would lie in the open half-space through the origin which is parallel to $H$, a contradiction.  Now we take $\S k={S'}^{(k)}\times \{0\}^{d-d'}$, and the origin is still in the convex hull of $\S k$.

Now, letting $\Delta^{d}$ denote a unit simplex centered at the origin (with $d+1$ points), we let 
\[
U=\bS k\cup \bigcup_{p\in \S k} \{(1.5)p+.1\cdot \Delta^d\}.
\]
So $U$ is a set of $1+k+(d+1)k$ points. (Figure \ref{f.dk} shows $U$ for the case $d=2,k=2$; note that in this case, $d'=1$.)

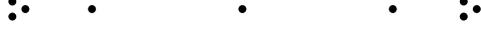
\begin{figure}[t]
\begin{center}
\begin{pdfpic}
\psset{unit=2cm,dotsize=3pt}
\begin{pspicture}(-1.6,-.3)(1.6,.3)
\SpecialCoor
\psdot(0,0)
\psdot(1,0)
\psdot(-1,0)

\rput(-1.5,0){
\psdot(-.0289,-.05)
\psdot(-.0289,.05)
\psdot(.0577,0)}

\rput(1.5,0){
\psdot(-.0289,-.05)
\psdot(-.0289,.05)
\psdot(.0577,0)}
\end{pspicture}
\end{pdfpic}
\end{center}
\caption{\label{f.dk} A configuration for forcing degree $2$ in 2-dimensions.}
\end{figure}

We now let $U_{\e,D}$ denote an $(\e,D)$ copy of $U$, for sufficiently small $\e>0$ and sufficiently large $D$.  Observe that since the origin is in the convex hull of $\S k$, the $k$ small copies of the $d$-simplex in $U$ ensure that the origin is in the \emph{interior} of the convex hall of $U$, and thus also in the interior of $U_{\e,D}$ for sufficiently small $\e$. 

Observe that (for large $D$)  the distance between any pair of points in an $U_{\e,D}$ is less than the minimum distance between $U_{\e,D}$ and $\ydn\stm U_{\e,D}$.  In particular, if $T$ denotes the minimum length spanning tree on $\ydn$, the subgraph $T[U_{\e,D}]$ induced by the points in $U_{\e,D}$ must be connected (and so a tree), or we could exchange a long edge for a short edge.  Moreover, the minimum length spanning tree on $T$ must restrict to a minimum length spanning tree on $U_{\e,D}$, and by construction, the point of $U_{\e,D}$ corresponding to the origin point in $U$ has degree $k$ in the MST on $U$.  Finally, no points in $\ydn\stm U_{\e,D}$ can be adjacent to the center of the star when $D$ is sufficiently large, by Observation \ref{o.convshield}.  Thus Observation \ref{o.findit} gives that $\a_{k,d}>0$ for $d+1\leq k\leq \tau'(d)$.

Finally, $\a_{1,d}>0$ is an immediate consequence of $\a_{3,d}>0$.\qed

Indeed, Theorem \ref{t.mstseq} follows immediately as well:
\subsection*{Proof of Theorem \ref{t.mstseq}}
Suppose $2\leq k<\tau'(d)$, and $T$ is a minimum spanning tree of $\ydn$ subject to the restriction that the maximum degree is $\leq k$.  By Observation \ref{o.findit} we have that there are $Cn$ $(\e,D)$ copies of the set $U$ from the previous proof, for some constant $C$, and from the argument above we see that each such copy $S_i$ will induce a (connected subtree) $T[S_i]$, which will have maximum degree at most $k$ in an instance of $\mst_k$. Replacing each $T[S_i]$ by the optimum $(k+1)$-star produces a spanning tree of maximum degree $k+1$, whose length is less by at least some constant $C' n$.  Rescaling by $t$ gives that the length difference is at least $C'n^{\frac{d-1}{d}}$.\qed

\begin{remark}
The same argument allows us to separate $\b^d_{\mst}$ from $\b^d_{Steiner}$ where the latter corresponds to the minimum length Steiner tree. We just need to use $(\e,D)$ copies of an equilateral triangle. We remark that adding the Steiner points corresponding to the Fermat points of the copies will reduce the tree length. The details can be left to the reader.
\end{remark}

We turn our attention now to 2-factors.  We begin with two very simple geometric lemmas:
\begin{lemma}\label{parallel}
Suppose that points $p,q,r,s$ satisfy
$$||p-q||, ||r-s||\geq \D\text{ and }||q-r||\leq \d,$$
where $\D\gg\d$ i.e $\D$ is sufficiently large with respect to $\d$.

Let $\th(x;y,z)$ denote the angle between the line segments $xy$ and $xz$. If $$\max\set{\th(p;q,s),\th(s;p,r)}\geq \D^{-1/3}$$ 
then
$$||p-s||\leq ||p-q||+||r-s||+\d-\frac{\D^{1/3}}4.$$
\end{lemma}
\begin{proof}
We have
$$||p-s||\leq ||p-q||\cos\th(p;q,s)+\d+||r-s||\cos\th(s;p,r).$$
Now use $\cos x\leq 1-x^2/3$ for $x\leq 1$.
\end{proof}
\begin{minipage}{5in}
\begin{lemma}
\label{o.4points}
Suppose that points $p_i,q_i,r_i,s_i,i=1,2$ satisfy
\beq{long2}
||p_i-q_i||,||r_i-s_i||\geq \D\quad \text{for }i=1,2
\eeq
and also that $q_1,r_1,q_2,r_2$ are contained in a ball of radius $\d$. Then there is a matching on $\set{p_1,p_2,s_1,s_2}$ whose total length is at most 
\begin{equation}\label{lbound}
||p_1-q_1||+||r_1-s_1||+||p_2-q_2||+||r_2-s_2||+4\d -\tfrac 12 \D.
\end{equation}
\end{lemma}
\end{minipage}
\begin{proof}
Without loss of generality we let the $q_i,r_i$ be within distance $\d$ of the origin, and then let $\theta(x,y)$ denote the angle between $x$ and $y$ via the origin that is less than or equal to $\p$. There are three possible pairings of the points $P=\{p_1,p_2,s_1,s_2\}$, and for at least one such pairing, $\theta(x,y)< \tfrac 1 2 \pi$ for one of the pairs.

Let us take $\{x,y\}$ and $\{w,z\}$ to be the pairs in such a pairing of $P$, with $\theta(x,y)\leq\tfrac 1 2 \pi$.  We let $T$ denote the triangle with vertices $x,y,0$, let $a,b,c$ denote the sidelengths, where $a$ is length of the side opposite 0, and $s$ denote the semi-perimeter $(a+b+c)/2$. Now $a\leq (b^2+c^2)^{1/2}$ and in fact
\begin{multline*}
b+c-a\geq b+c-(b^2+c^2)^{1/2}
=(b+c)\brac{1-\brac{1-\frac{2bc}{(b+c)^2}}}\\
\geq \frac{bc}{b+c}
\geq \frac12\min\set{b,c}
\geq \frac12 \D.
\end{multline*}
Thus we find a pairing of $P$ for which the total length is at most $||p_1|| +||p_2|| + ||s_1|| + ||s_2||-\tfrac 1 2 \D$, and we will be done after applying the triangle inequality four times and using the fact that $||q_i||,||r_i||\leq \d$ for $i=1,2$.
\end{proof}
\begin{figure}[t]
\begin{center}
\begin{pdfpic}
\psset{unit=.5cm,dotsize=2pt}
\begin{pspicture}(-5,-2)(5,2)
\SpecialCoor
\rput{90}(-3,0){
\psframe(-5,-2)(5,2.5)
\psdot(-.2,0)
\psdot(-.2,.4)
\psdot(.1,.4)
\psdot(.2,.1)

\pscircle[linestyle=dashed](0,.2){.4}

\psdot(-4,2)
\psdot(-3.8,-1.4)
\psdot(3.7,1.2)
\psdot(3.9,1.8)

\psline(-5,1)(-4,2)(-.2,.4)(.1,.4)(3.9,1.8)(5,2)
\psline(-5,-1)(-3.8,-1.4)(-.2,0)(.2,.1)(3.7,1.2)(5,1)
\psline[linestyle=dashed](-4,2)(3.9,1.8)
}

\rput{90}(3,0){
\psframe(-5,-2)(5,2.5)
\psdot(-.2,0)
\psdot(-.2,.4)
\psdot(.1,.4)
\psdot(.2,.1)

\pscircle[linestyle=dashed](0,.2){.4}

\psdot(-4,2)
\psdot(-3.8,-1.4)
\psdot(3.7,1.2)
\psdot(3.9,1.8)

\psline(-5,1)(-4,2)(3.9,1.8)(5,2)
\psline(-5,-1)(-3.8,-1.4)(3.7,1.2)(5,1)
\psline(-.2,0)(-.2,.4)(.1,.4)(.2,.1)(-.2,0)
}

\end{pspicture}
\end{pdfpic}
\end{center}
\caption{\label{f.shortcut} When not all pairs are nearly straight the old 2-factor (left) can be shortened to a new one (right). (The dashed circle of radius $\e$ encloses $g+1=4$ points.)}
\end{figure}
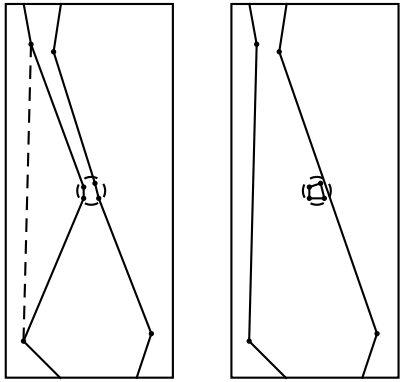

\subsection*{Proof of Theorem \ref{t.girthseq}}
Let $F_{g+1}$ be a minimum length 2-factor in $\ydn$ whose cycles all have length $\geq g+1$.  We let $U_{\e,D}\subset \ydn$ denote any set of $g$ points of radius $\e$ and at distance $D$ from $\ydn\stm U_{\e,D}$.  Note that Lemma \ref{o.findit} implies that there are a linear number of copies of such sets.  We now define $V_{\e,D,F}$ as a collection of three instances $U_1,U_2,U_3$ of $U_{\e,D}$, centered at the vertices of an equilateral triangle of sidelength $2D$, and lying at distance $\D$ from $\ydn\stm V_{\e,D,\D}$; we will take $D$ large relative to $\e$ and $\D$ large relative to $D$. 

We call a multiset of edges a \emph{2-matching} if every vertex is incident with exactly 2 (not necessarily distinct) edges in the multiset.  This is the same as a 2-factor, except that it can contain ``2-cycles'', which consist of a single edge included twice.  

We will begin by showing how to give a constant-factor shortening of the 2-factor $F_{g+1}$ to a 2-matching $F$, without being careful to avoid creating cycles of length shorter than $g$. In particular, we prove the following lemma:
\begin{lemma}\label{l.V}
There is an absolute constant $\d$ such that for suitable choices of $\e\ll D\ll \D$, any instance of $V=V_{\e,D,\D}$ allows a modification $F$ of $F_{g+1}$ so that
\begin{enumerate}
\item $F$ is a 2-matching;
\item $F$ has weight at least $\d$ less than the length of $F_{g+1}$;
\item \label{p.Vgirth} Cycles of $F$ lying entirely in $V$ have length $\geq g$;
\item $F$ is a local modification of $F_{g+1}$, in the sense that any edges of $F_{g+1}$ disjoint from $V$ are still present in $F$.
\end{enumerate}
\end{lemma}
Again, Lemma \ref{o.findit} implies that there are a linear number of instances of $V_{\e,D,\D}$ in $\ydn$.  In particular, this lemma would be sufficient to argue that $\btfg<\btF{g+1}$, except that $F$ may not have girth $g$.

\begin{proof}[Proof of Lemma \ref{l.V}]
For $U_i=U_{\e,D}$ in $V$, there are (at least 2) edges in $F_{g+1}$ from $\ydn\stm U_i$ to $U_i$, since $g+1>g=|U_i|$.  We can pair these edges so that each pair lies on a common cycle of $F_{g+1}$, and so that the two edges in a pair are joined in $F_{g+1}$ by a path through (possibly just 1 point of) $U_i$.    Similarly, we can pair edges between $V$ and $\ydn \stm V$.  (Some pairs for $V$ may also be pairs for a $U_i$, others may not.)

Now, by choosing $D$ large relative to $\e$, we can assume that each pair of edges for a $U_i$ is \emph{nearly straight}, in the sense that the angle between the endpoints of the edges in $\ydn$ via any point in $U_{\e,D}$ is close to $\pi$; otherwise, we can modify $F_{g+1}$ by including all edges of some $g$-cycle through $U_i$, and shortcutting each pair of edges between $\ydn\stm U_i$ and $U_i$ with a single edge between the endpoints in $\ydn\stm U_i$. (Figure \ref{f.shortcut}.)  The result has length smaller by a constant $\d=\Omega(D^{1/3})$, see Lemma \ref{parallel}.  To ensure condition \eqref{p.Vgirth} for $F$, we must now also shortcut all remaining pairs of edges between $V$ and $\ydn\stm V$, delete any edges in $V\stm U_i$, and then add $g$-cycles to the remaining $U_j$'s.  (This step adds length which can be made arbitrarily small by decreasing $\e$.)

\begin{figure}[t]
\begin{center}
\begin{pdfpic}
\psset{unit=1cm,dotsize=3pt}
\begin{pspicture}(-1.5,-1)(2.5,2.866)
\SpecialCoor
\rput{0}(1,0){
\psdot(.1,0)
\psdot(-.1,0)
\psdot(0,.173)
}

\rput{0}(-1,0){
\psdot(.1,0)
\psdot(-.1,0)
\psdot(0,.173)
}
\rput{0}(0,1.73){
\psdot(.1,0)
\psdot(-.1,0)
\psdot(0,.173)
}

\psline(-1.5,-.866)(-.9,0)(-1.1,0)(-1,.173)(.1,1.73)(-.1,1.73)(0,1.903)(.5,2.598)

\psline(.5,-.866)(.9,0)(1.1,0)(1,.173)(2.5,2.598)

\end{pspicture}
\end{pdfpic}
\end{center}
\caption{\label{f.V} An instance of $V_{\e,D,\D}$ (here for $g=2$, $d=2$).  When all pairs of edges entering/leaving $U_i$'s are nearly straight, we must have at least 2 pairs of edges entering/leaving $V$, as shown here.}
\end{figure}
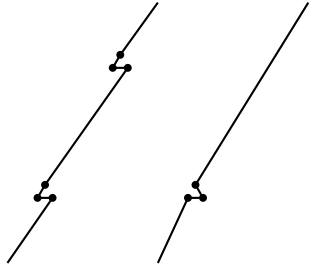

We may also assume that each $U_i$ has only a single pair of edges.  Otherwise, if there are two different pairs, we delete the edges in the two pairs, use Lemma \ref{o.4points} to add a pair of edges among the 4 outside endpoints of the pairs of total weight which is less than the total weight of the pairs by a constant (note that we may have created a 2-cycle if one of these edges was already present, which is why $F$ is only a 2-\emph{matching}), shortcut all other remaining pairs between $V$ and $\ydn$, delete all edges within $V$, and add $g$-cycles to each $U_i$.  For sufficiently small $\e$, we get a constant length improvement. 

 Thus we may assume that each $U_i$ in $V$ has a single pair, and that the pair for each $U_i$ in $V$ is nearly straight.  The crucial point is that this implies that there must be at least \emph{two} pairs of edges joining $V$ to $\ydn\stm V$: since, e.g., edges joining $U_1$ to $U_2$ and $U_1$ to $U_3$ would not be nearly straight. Therefore at least one of the $U_i$'s has no edges to the other $U_i$'s.  (See Figure \ref{f.V}.)  We conclude, as in the previous paragraph, by deleting the edges in the two pairs, using Lemma \ref{o.4points} to find a pair of edges among the 4 outside endpoints of the pairs of total weight which is less than the total weight of the pairs by a constant, shortcutting all other remaining pairs between $V$ and $\ydn$, deleting all edges within $V$, and adding $g$-cycles to each $U_i$. 
\end{proof}

We must now address unintentional problems of girth.  (Notice that, in shortcutting edges, we may have left behind short cycles---in particular, any 2-cycles must be eliminated.)  To this end, we say that $V=V_{\e,D,\D}$ is $\e$-\emph{surrounded} if the set $\cN_V$ of points of $\ydn\stm V$ within distance $3\D$ of $V$ has the properties that: (1) each $x\in \cN_V$ lies within distance $\e$ of the sphere $S$ of radius $2\D$ centered at the center of $V$, and (2) each $x\in S$ lies within $\e$ of $\cN_V$.  (Essentially, $\cN_V$ is an approximation to an $\e$-net on $S$, which surrounds $V$).  Lemma \ref{o.findit} implies that there are a linear number of $\e$-surrounded $V$'s, and additionally, a linear number of $\e$-surrounded sets $V$ satisfying the requirements in the previous paragraph (each $U_i$ has a single-pair of edges to the rest of $\ydn$, etc.).

We now show that if $V$ is $\e$-surrounded, then there is an constant $C_{g,\e}$, which can be made arbitrarily small by decreasing $\e$, such that there is a 2-factor $F'$ such that:
\begin{enumerate}[(A)]
\item \label{Fweight} $F'$ has total weight $w(F')\leq w(F_{g+1})+C_{g,\e}$,
\item \label{Ngirth} every cycle in $F'$ 
is still of length $\geq g+1$,
\item \label{Surrounded} All edges in $F'$ incident with $V$ either lie in $V$ or 
intersect $\cN_V$.
\end{enumerate} 

To produce $F'$ from $F_{g+1}$, we consider each edge $e=\set{u,v}$ from $V$ to $\ydn\stm (\cN_V\cup V)$ which does not intersect $\cN_V$, and 
\begin{enumerate}
\item Locate a point $x$ in $\cN_V$ within distance $\e$ of a point $w$ on the edge $e$. Let $C=(x=x_1,x_2,\ldots x_k,x_{k+1}=x_1)$ be the cycle of $F_{g+1}$ that contains $x$.  If $u=x_i$ for some $i$, then we choose the cycle orientation so that $v=x_{i-1}$.

\item \label{S.add} Add the edges $\set{u,x_1}$, $\set{x_k,v}$ to the 2-factor and delete the edges $e$ and $\set{x_1,x_k}$.
\end{enumerate}
This ensures (C) and the change in cost for this one substitution is 
\begin{align*}
&||x_1-u||+||x_k-v||-||x_1-x_k||-||v-w||-||u-w||\\
\leq& ||x_1-u||+||x_1-w||-||u-w||\\
\leq& 2||x_1-w||.
\end{align*}
Thus dealing with all edges from $V_{\e,D,\D}$ to $\ydn\stm V_{\e,D,\D}$ increases the cost by at most $12g\e$, since there are $3g$ points in $V$ and hence at most $6g$ edges from $V_{\e,D,\D}$ to $\ydn\stm V_{\e,D,\D}$.

After this, any cycle in $F'$ but not in $F_{g+1}$ must contain an edge added in Step \eqref{S.add}.  But either $u,v\notin \{x_1,\dots,x_k\}$, in which case the length of this cycle is at least $k+2\geq g+3$, or else $u=x_i,v=x_{i+1}$ and this cycle is $x_1,x_2,\dots,x_{i-1}x_kx_{k-1}\cdots x_{i}x_1$ and so has length $k\geq g+1$.

\bigskip

We are now prepared to find a 2-factor $F_g$ whose weight is smaller than $F_{g+1}$ by a constant factor.  For some small constant $c$, we have that there are at least $cn$ instances of $\e$-surrounded $V=V_{\e,D,\D}$'s.  We take these instances as $V_1,V_2,\dots, $ in any order, and beginning with $F=F_{g+1}$ and for each $i=1,2,\dots,$ we
\begin{enumerate}[(i)]
\item Find $F'$ for $V_i$ as above (with weight increase $C_{g,\e}$ which we make arbitrarily small) 
\item Apply Lemma \ref{l.V} to shorten $F'$ at $V_i$ to $F_0$ with a constant weight improvement
\item \label{s.merge} At an arbitrarily small cost, modify $F_0$ to a 2-factor $F_0'$ which has girth $g$, by merging cycles intersecting the net $\cN_{V_i}$, and set $F=F_0'$ (explanation is below).
\end{enumerate}

In particular, to carry out Step \eqref{s.merge}, note that any cycle $C$ of length $<g$ in $F_0$ includes a point $x$ of $\cN_V$, and we can merge $C$ with the cycle through a point $y$ within $2\e$ of $x$, at an additional cost of $\leq 2\e$: We join $x$ and $y$, delete edges $\set{x,x'}$ and $\set{y,y'}$ incident with each in the previous 2-factor and replace them by $\set{x,y},\set{x',y'}$ at a cost of
$$||x-y||+||x'-y'||-||x-x'||-||y-y'||\leq  2||x-y||.$$

After applying Steps (i)--(iii) for each $V\in \cV$, the result is a 2-factor $F_g=F$ of girth $g$, whose total weight is smaller than the total weight of $F_{g+1}$ by a constant factor.
\qed
\bigskip

The proof of the counterpoint Theorem \ref{t.2flimit} will be given in Section \ref{s.lim}.  For now we consider matchings.  In fact, Theorem \ref{t.mmnottsp} can be viewed as a consequence of Theorem \ref{t.hk}, via Proposition 5 of \cite{GB}.  However, we also give a short self-contained proof.

\subsection*{Proof of Theorem \ref{t.mmnottsp}}
We define the Euclidean functional $\tmm(X)$ as the minimum length union of two matchings on $X$.  Note that we make no requirement of disjointness and that we trivially have that $\tmm(X)=2\cdot \mm(X)$ for all $X$.  On the other hand, a TSP through $X$ can be viewed as a (near)-union of two matchings (alternating edges around the tour, leaving one vertex unmatched if $n$ is odd).  Our aim will be to give a constant factor improvement to the union of a pair of matchings given by the TSP, to show that $\tmm(\ydn)$ is asymptotically less than $\tsp(\ydn)$.  To this end, we let $M_1$ and $M_2$ denote a pair of matchings derived from the minimum length TSP.

We let $U_{\e,D}$ denote a set of two points separated by distance at most $\e$ and at distance at least $D$ from all other points of $\ydn$, and let $V_{\e,D,F}$ denote a collection of 5 instances $U_1,\dots,U_5$ of $U_{\e,D}$, centered at the vertices of a regular pentagon of sidelength $2D$, such that all other points of $\ydn$ are at distance $\geq F$ from this set.  As before, Lemma \ref{o.findit} gives that there are a linear number of instances of $V_{\e,D,F}$ for any fixed $F,$ $D,$ and $\e>0$.  Moreover, as before, if we have a linear number of instances $U_{\e,D}$ in which a pair of edges of a matching leaves $U_{\e,D}$ and is not nearly straight, then we can make a constant improvement to the matching, by joining the two points of $U_{\e,D}$ and shortcutting the outside endpoints of the edges leaving $U_{\e,D}$ with a single edge.

Since $M_1$ and $M_2$ are disjoint, the pigeonhole principle gives that for some $s\in \{1,2\}$ and at least three of the $U_i$'s in any $V_{\e,D,F}$, the pair of points in $U_i$ is omitted from $M_s$.  In particular, we may assume without loss of generality that we have a linear number of $V_{\e,D,F}$'s for which the set $I$ of indices $i$ for which the points in $U_i$ are unmatched in $M_1$ has cardinality $|I|\geq 3$.  Moreover, from the previous paragraph, there must be a linear number of such $V_{\e,D,F}$'s which also have the property that the pair edges leaving the $U_i$, $i\in I$ is nearly straight.  In particular, as the point sets $U_i$ $(i\in I)$ are not nearly collinear, we must have as in the previous proof that there are (at least) 2 pairs of edges entering and leaving $V_{\e,D,F}$.  We conclude by applying Lemma \ref{o.4points} (with $2\e$, say) to get a constant factor improvement a linear number of times.\qed

\subsection*{Proof of Theorem \ref{t.hfac}}
It suffices to show that for fixed $r\geq 1$, there are connected graphs $H$ with $r\cdot |V(H)|$ edges for which the constant $\b^d_H$ is arbitrarily large, which we show by demonstrating that $\b^d_T$ can be arbitrarily large even just over trees $T$.  To this end, we let $T_k$ be the tree on $k+1$ vertices which has $k$ leaves.

Given any large constant $u=t/m$ for some integer $m$, we decompose the $[0,t]^d$ cube with $m^d$ subcubes
of side $u$. Now the number of points in each subcube is binomially distributed with mean $u^d$. Let a point in $\ydn$ be {\em good} if the subcube $S_\a$ that it lies in has at least $(1-\e)u^d$ members of $\ydn$ and the total number of points in the $\leq3^d$ subcubes that touch $S_\a$ contain at most $(1+\e)(3u)^d$ members of $\ydn$, where $\e=\frac{1}{10k}$. Assuming that $u$ is sufficiently large, the Chernoff bounds imply that a member of $\ydn$ is good with probability at least $1-\e/2$. Thus the expected number of good points in $\ydn$ is at least $(1-\e/2)n$.  Now the Chernoff bounds can be used to show that the number of members of $\ydn$ in any subcube is a.s. $O(\log n)$ and therefore, changing one point only changes the number of good points by $O(\log n)$ a.s. A fairly simple modification of McDiarmid's inequality now implies that a.s. $(1-\e)n$ of the members of $\ydn$ are good.

Since $\approx n/(k+1)$ points must have degree $k$ in a $T_k$ factor of $\ydn$, we have that there are at least $n/(2k)$ good points which have degree $k$. Now let $k=2(3u)^d$. Then a.s. a $T_k$ factor has length at least $\frac{n}{2k}\cdot\frac{(1-\e)k}{2}\cdot u>\frac{un}{5}$.

Rescaling the $[0,t]^d$ cube by a factor of $t$ gives that the minimum $T_k$ factor has length at least $\frac 1 5 u n^{\frac {d-1}{d}}$, and here $\frac u 5$ is an arbitrarily large constant.\qed

\subsection*{Proof of Theorem \ref{t.hk}}
We begin with some general observations regarding the shortest TSP through Euclidean point sets:
\begin{observation}\label{2meansstraight}
Suppose that $S_{\e,D}$ is an $(\e,D)$ copy of a fixed set $S$ for fixed $\e$ and sufficiently large $D$, and that at least 2 pairs of edges of a shortest TSP tour $\cL$ join $S_{\e,D}$ to $V\stm S_{\e,D}$.  Then the pairs are nearly straight (i.e., the angle for each pair is arbitrarily close to $\pi$ as $\e\to 0$, and $k,D\to \infty$).
\end{observation}
(Here the edges of $\cL$ joining $S_{\e,D}$ to $V\stm S_{\e,D}$ are \emph{paired} such that the endpoints of the edges in a pair which are inside $S_{\e,D}$ are joined by a portion of $\cL$ lying entirely in $S_{\e,D}$.)
\begin{proof}
Otherwise, we shortcut the edge pair which is not nearly straight to obtain a constant improvement (which, for a fixed angle, can be made large by increasing $D$).  The tour portion between one of the other edge pairs is modified to cover any vertices of $S_{\e,D}$ which are now missed by the tour, at an increased cost which does not depend on $D$.  
\end{proof}
\begin{observation} \label{nothree}
Suppose that $S_{\e,D}$ is an $(\e,D)$ copy of any fixed set $S$ for fixed $\e$ and sufficiently large $D$.  Then there are at most 2 pairs of edges in a shortest TSP tour which join $S_{\e,D}$ to $V\stm S_{\e,D}$.
\end{observation}
\begin{proof} 
Let $\cL$ denote a shortest TSP tour, and suppose there are three pairs $(e_1,e_2)$ $(f_1,f_2)$ and $(g_1,g_2)$ of edges in $\cL$ between $S_{\e,D}$ and $V'=V\stm S_{\e,D}$.  We let $x_1,x_2,y_1,y_2,z_1,z_2$ denote the endpoints of the edges $e_1,e_2,f_1,f_2,g_1,g_2$, respectively, which lie in $V'$, and we suppose, without loss of generality, that the pairs $x_1,y_1$, $y_2,z_2$, and $z_1,x_2$, respectively, are joined by paths in $V''\cap \cL$, for $V''=V\setminus \{x_1,x_2,y_1,y_2,z_1,z_2\}$.

We now modify $\cL$ as follows.
\begin{enumerate}
\item  We remove the edges $e_1,e_2,f_1,f_2,g_1,g_2$.
\item We add new edges between the pairs $(x_2,y_2)$ and $(y_1,z_1)$.
\item \label{p.newpath} We add a path which travels from $x_1$ to the set $S_{\e,D}$ (visiting every vertex of the set) and then to $z_2$.
\end{enumerate}
Observe that the result is a new TSP tour; it remains to estimate the change in cost.  The path $P$ added in part \ref{p.newpath} has Euclidean length $\ell(P)$ at most $\ell(e_1-g_2)+C_{S,\e,D}$, where we are viewing the edges as vectors from their point in $V''$ to their point in $S_{\e,D}$, and where $C_{S,\e}$ is a constant depending only on $S,\e$.  (For example, we can take $C_{S,\e}=\tsp(S_{\e,D})+\diam(S).$)

Similarly, the edge $(x_2,y_2)$ has length at most $\ell(e_2-f_2)+C_{S,\e,D}$, and the edge $(y_1,z_1)$ has length at most $\ell(f_1-g_1)+C_{S,\e,D}$.  Applying the triangle inequality to the three lengths immediately gives that our new tour has length at most $\tsp(\xdn)+3C_{S,\e,D}$.  In fact, we should be hoping to do better: If $\ell(e_2-f_2)$ is within a constant of $\ell(e_2)+\ell(f_2)$ as $D$ grows large, then the points $x_2,y_2$ are constrained to be nearly antipodal about center $c$ of the set $S_{\e,D}$ (with angle tending to $\pi$ as $D\to \infty$).  Similarly, we have that $\ell(f_1-g_1)$ is far from $\ell(f_1)+\ell(g_1)$ unless $y_1,z_1$ are nearly antipodal.  

Thus if we have not achieved a contradiction by shortening the tour, then, $x_2,y_2$ are nearly antipodal, and $y_1,z_1$ are nearly antipodal.   Observation \ref{2meansstraight} means that the pairs $\{x_1,x_2\}$, $\{y_1,y_2\}$, and $\{z_1,z_2\}$ are nearly antipodal as well.  Thus, in particular, $x_1$ and $z_1$ are \emph{not} nearly antipodal, and so we can produce a shortening of the original TSP tour by instead:
 \begin{enumerate}
\item Removing the edges $e_1,e_2,f_1,f_2,g_1,g_2$;
\item Adding new edges between the pairs $(x_1,z_1)$ and $(x_2,z_2)$;
\item Adding a path which travels from $y_1$ to the set $S_{\e,D}$ (visiting every vertex of the set) and then to $y_2$.\qedhere
\end{enumerate}
\end{proof}

We now consider the $d=2$ case of Theorem \ref{t.hk}.  We let $S^k$ be a set consisting of $k$ equally spaced points on a unit circle centered at the origin, $2k$ equally spaced points at the radius 4 circle centered at the origin, and the points $(2,0)$ and $(-2,0)$ (See Figure \ref{f.hkG}.  The particular ratios $2k:k$ and $4:2:1$ are chosen just to make a clean figure.).  We will argue that if $\e$ is sufficiently small and $D,k$ are sufficiently large, each instance of an $(\e,D)$ copy of $S^k$ allows us to locally modify an instance of the TSP so that it is still a solution to the $\hk$ linear program, but is shorter by some additive constant.

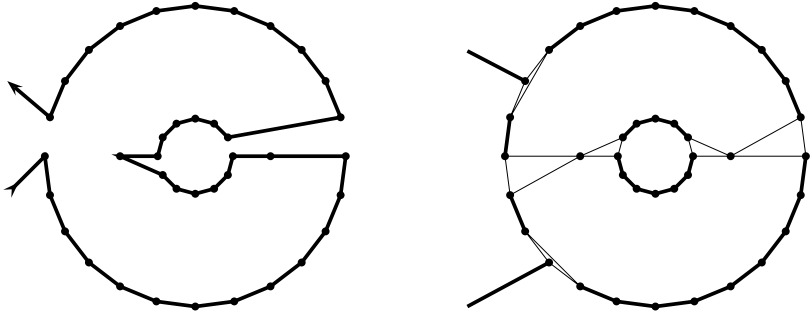
\begin{figure}[t]
\begin{center}
\begin{pdfpic}
\psset{unit=.5cm,dotsize=3pt,linewidth=1.4pt}
\begin{pspicture}(-5,-4)(5,4)
\SpecialCoor
\psdot(1;0)
\psdot(1;30)
\psdot(1;60)
\psdot(1;90)
\psdot(1;120)
\psdot(1;150)
\psdot(1;180)
\psdot(1;210)
\psdot(1;240)
\psdot(1;270)
\psdot(1;300)
\psdot(1;330)

\psdot(2;0)
\psdot(-2,0)

\psdot(4;0)
\psdot(4;30)
\psdot(4;60)
\psdot(4;90)
\psdot(4;120)
\psdot(4;150)
\psdot(4;180)
\psdot(4;210)
\psdot(4;240)
\psdot(4;270)
\psdot(4;300)
\psdot(4;330)

\psdot(4;15)
\psdot(4;45)
\psdot(4;75)
\psdot(4;105)
\psdot(4;135)
\psdot(4;165)
\psdot(4;195)
\psdot(4;225)
\psdot(4;255)
\psdot(4;285)
\psdot(4;315)
\psdot(4;345)

\psline{>->}(-5,-1)(4;180)(4;195)(4;210)(4;225)(4;240)(4;255)(4;270)(4;285)(4;300)(4;315)(4;330)(4;345)(4;0)(2;0)(1;0)(1;330)(1;300)(1;270)(1;240)(1;210)(2;180)(1;180)(1;150)(1;120)(1;90)(1;60)(1;30)(4;15)(4;30)(4;45)(4;60)(4;75)(4;90)(4;105)(4;120)(4;135)(4;150)(4;165)(-5,2)

\end{pspicture}\hspace{1cm}
\begin{pspicture}(-5,-4)(5,4)
\SpecialCoor
\psdot(1;0)
\psdot(1;30)
\psdot(1;60)
\psdot(1;90)
\psdot(1;120)
\psdot(1;150)
\psdot(1;180)
\psdot(1;210)
\psdot(1;240)
\psdot(1;270)
\psdot(1;300)
\psdot(1;330)

\psdot(2;0)
\psdot(-2,0)

\psdot(4;0)
\psdot(4;30)
\psdot(4;60)
\psdot(4;90)
\psdot(4;120)
\psdot(4;150)
\psdot(4;180)
\psdot(4;210)
\psdot(4;240)
\psdot(4;270)
\psdot(4;300)
\psdot(4;330)

\psline[linewidth=.2pt](4;135)(4;150)(4;165)(4;135)
\psline(4;165)(4;180)
\psline[linewidth=.2pt](4;195)(2;180)(4;180)(4;195)
\psline[linewidth=.2pt](2;180)(1;180)(1;150)(2;180)
\psline(4;195)(4;210)
\psline[linewidth=.2pt](4;210)(4;240)(4;225)(4;210)
\psline[linewidth=.2pt](2;0)(1;30)(1;0)(2;0)(4;15)(4;0)(2;0)

\psline(4;135)(4;120)(4;105)(4;90)(4;75)(4;60)(4;45)(4;30)(4;15)
\psline(4;0)(4;345)(4;330)(4;315)(4;300)(4;285)(4;270)(4;255)(4;240)

\psline(1;0)(1;330)(1;300)(1;270)(1;240)(1;210)(1;180)
\psline(1;150)(1;120)(1;90)(1;60)(1;30)

\psline(4;150)(-5,2.8)
\psline(4;225)(-5,-4)

\psdot(4;15)
\psdot(4;45)
\psdot(4;75)
\psdot(4;105)
\psdot(4;135)
\psdot(4;165)
\psdot(4;195)
\psdot(4;225)
\psdot(4;255)
\psdot(4;285)
\psdot(4;315)
\psdot(4;345)
\end{pspicture}
\end{pdfpic}
\end{center}
\caption{\label{f.hkG} The set $S^{12}$.  Such sets are good for the HK-bound, but bad for the tour. (Thick lines have weight $1$, while thin lines have weight $\frac 1 2$.)} 
\end{figure}

To this end, let $\cL$ be some shortest tour, and $S^k_{\e,D}$ be an $(\e,D)$-copy of $S^k$.  Observation \ref{nothree} implies that $\cL$ enters (and leaves) $S^k_{\e,D}$ either once or twice, for sufficiently large $D$.
\subsubsection*{Case 1: The tour enters $S^k_{\e,D}$ once}
We let $x_1,x_2$ denote the vertices in $V'=V\stm S^k_{\e,D}$ which are adjacent to vertices in $S^k_{\e,D}$.  We let $\cL^\circ$ be the portion of $\cL$ between $x_1$ and $x_2$.  Then we have that 
\begin{equation}\label{case1cost}
\ell(\cL^\circ)\geq \dist(x_1,S^k_{\e,D}) + \dist(S^k_{\e,D},x_2) + 10\pi+6+K-o(1),
\end{equation} where we are using $o(1)$ to denote a function which tends to 0 as $\e\to 0$, $k\to \infty$, and $D\to \infty$ simultaneously, and $K$ is an absolute constant (in fact, $K$ can be 2).  To see this lower bound, observe that the tour must cover the outer circle ($\approx 8\pi$), the inner circle ($\approx 2\pi$), and must cross the gap between the inner and outer circles twice ($2\cdot 3=6$).  Finally, the tour must also spend more (bounded below by some constant $K$) to cover one of the two ``gap'' vertices in $S^k_{\e,D}$.  To see that the tour can not cover both gap vertices in gap crossings while crossing the gap only twice, let $a_1,a_2$ denote the first and last vertices of $\cL^\circ$ on the inner circle.  Either $a_1,a_2$ lie at the two ends of the inner circle close to the gap vertices, in which case the tour spends an additive constant $K$ more than $2\pi-o(1)$ to cover the entire inner circle, or, say, $a_1$ lies at distance $1+K$ from the gap vertex to which it is joined, incurring an extra cost $K$ again.

\smallskip

We now modify the portion $\cL^\circ$ so that the result is still a solution to the Held-Karp LP, but is shorter by some additive constant.  We let $y_1,y_2$ and $z_1,z_2$ be pairs of points on the outer circle which are closest to the gap vertices $g_1$ and $g_2$, respectively, and similarly let $a_1,a_2$ and $b_1,b_2$ be points of the inner circle which are closest to the first and second gap vertex, respectively.  We join all pairs among each of the triples $y_1,y_2,g_1$, $a_1,a_2,g_1$, $z_1,z_2,g_2$, $b_1,b_2,g_2$ by edges of weight $\frac 1 2$ (Figure \ref{f.hkG}).  Next, we let $\a_1,\a_2,\a_3$ be a consecutive triple of points on the outer circle which is as close as possible to $x_1$ (but disjoint from the set $\{y_1,y_2,z_1,z_2\}$), and $\b_1,\b_2,\b_3$ be a consecutive triple of points on the outer circle which is as close as possible to $x_2$ (but disjoint from the set $\{y_1,y_2,z_1,z_2,\a_1,\a_2,\a_3\}$).  We join all pairs among each of the triples $\a_1,\a_2,\a_3$ and $\b_1,\b_2,\b_3$ with edges of weight $\frac 1 2$.  We join $\a_2$ and $\b_2$ to $x_1$ and $x_2$, respectively, by edges of weight 1.   Finally, using edges of weight 1, we join all consecutive pairs of points on each circle which were not already joined (by edges of weight $\frac 1 2$).  The result for $S^{12}$ is shown on the right-hand side of Figure \ref{f.hkG}.  As $k,D$ grow large and $\e$ grows small, the total cost of this is
\[
\dist(x_1,S^k_{\e,D}) + \dist(S^k_{\e,D},x_2) + 10\pi+6+o(1),
\]
thus it suffices to check that what we have given is indeed a valid instance of the Held-Karp LP.  We carry out this verification below, after describing the second case of the construction.

\subsubsection*{Case 2: The tour enters $S^k_{\e,D}$ twice}
We let $x^1_1,x^1_2$ and $x_1^2,x_2^2$ denote the two pairs of entry/exit points in $V'=V\stm S^k_{\e,D}$.  We let $\cL^1$ be the portion of $\cL$ between $x^1_1$ and $x^2_2$, $\cL^2$ be the portion between $x^2_1$ and $x^2_2$, and let $\cL^\circ=\cL^1\cup \cL^2$.  Then we have that
\begin{equation}\label{case2cost}
\ell(\cL^\circ)\geq \sum_{\substack{i=1,2\\j=1,2}}\dist(x^i_j,S^k_{\e,D}) +10\pi+6+K-o(1),
\end{equation}
 where $K$ is again an absolute constant (and where again, in fact, $K$ can be 2).  Again, $10\pi+6$ is needed to cover both circles, and transition to the inner circle and back.  If both $\cL^1$ and $\cL^2$ visit the inner circle this gives an extra cost of $\approx 6$ for the transitions, so we assume that $\cL^1$ is the only portion to visit the inner circle.  But the argument from the previous case shows that $\cL^1$ cannot cover the entire inner circle and visit both gap vertices without incurring an additive constant extra cost.

\begin{figure}[t]
\begin{center}
\begin{pdfpic}
\psset{unit=.5cm,dotsize=3pt,linewidth=1.4pt}
\begin{pspicture}(-5,-4)(5,4)
\SpecialCoor
\psdot(1;0)
\psdot(1;30)
\psdot(1;60)
\psdot(1;90)
\psdot(1;120)
\psdot(1;150)
\psdot(1;180)
\psdot(1;210)
\psdot(1;240)
\psdot(1;270)
\psdot(1;300)
\psdot(1;330)

\psdot(2;0)
\psdot(-2,0)

\psdot(4;0)
\psdot(4;30)
\psdot(4;60)
\psdot(4;90)
\psdot(4;120)
\psdot(4;150)
\psdot(4;180)
\psdot(4;210)
\psdot(4;240)
\psdot(4;270)
\psdot(4;300)
\psdot(4;330)

\psdot(4;15)
\psdot(4;45)
\psdot(4;75)
\psdot(4;105)
\psdot(4;135)
\psdot(4;165)
\psdot(4;195)
\psdot(4;225)
\psdot(4;255)
\psdot(4;285)
\psdot(4;315)
\psdot(4;345)

\psline{>-}(-5,-.5)(4;180)(4;195)(4;210)(4;225)(4;240)(4;255)(4;270)(4;285)(4;300)(4;315)(4;330)(4;345)(4;0)(2;0)(1;0)(1;330)(1;300)(1;270)(1;240)(1;210)(2;180)
\psline{->}(2;180)(1;180)(1;150)(1;120)(1;90)(1;60)(1;30)(4;15)(5;12)

\psline{>->}(5,2.5)(4;30)(4;45)(4;60)(4;75)(4;90)(4;105)(4;120)(4;135)(4;150)(4;165)(-5,1)

\end{pspicture}\hspace{1cm}
\begin{pspicture}(-5,-4)(5,4)
\SpecialCoor
\psdot(1;0)
\psdot(1;30)
\psdot(1;60)
\psdot(1;90)
\psdot(1;120)
\psdot(1;150)
\psdot(1;180)
\psdot(1;210)
\psdot(1;240)
\psdot(1;270)
\psdot(1;300)
\psdot(1;330)

\psdot(2;0)
\psdot(-2,0)

\psdot(4;0)
\psdot(4;30)
\psdot(4;60)
\psdot(4;90)
\psdot(4;120)
\psdot(4;150)
\psdot(4;180)
\psdot(4;210)
\psdot(4;240)
\psdot(4;270)
\psdot(4;300)
\psdot(4;330)

\psline[linewidth=.2pt](4;135)(4;150)(4;165)(4;135)
\psline(4;165)(4;180)
\psline[linewidth=.2pt](4;195)(2;180)(4;180)(4;195)
\psline[linewidth=.2pt](2;180)(1;180)(1;150)(2;180)
\psline(4;195)(4;210)
\psline[linewidth=.2pt](4;210)(4;240)(4;225)(4;210)
\psline[linewidth=.2pt](2;0)(1;30)(1;0)(2;0)(4;15)(4;0)(2;0)

\psline[linewidth=.2pt](4;30)(4;45)(4;60)(4;30)
\psline[linewidth=.2pt](4;-15)(4;-30)(4;-45)(4;-15)

\psline(4;135)(4;120)(4;105)(4;90)(4;75)(4;60)
\psline(4;30)(4;15)
\psline(4;0)(4;345)
\psline(4;315)(4;300)(4;285)(4;270)(4;255)(4;240)

\psline(4;330)(5,-2)
\psline(4;45)(5,3.5)

\psline(1;0)(1;330)(1;300)(1;270)(1;240)(1;210)(1;180)
\psline(1;150)(1;120)(1;90)(1;60)(1;30)

\psline(4;150)(-5,2)
\psline(4;225)(-5,-4)

\psdot(4;15)
\psdot(4;45)
\psdot(4;75)
\psdot(4;105)
\psdot(4;135)
\psdot(4;165)
\psdot(4;195)
\psdot(4;225)
\psdot(4;255)
\psdot(4;285)
\psdot(4;315)
\psdot(4;345)
\end{pspicture}
\end{pdfpic}
\end{center}
\caption{\label{f.hkG2} An example of the set $S^{12}$ with two passes of a tour. (Thick lines have weight $1$, while thin lines have weight $\frac 1 2$.)} 
\end{figure}
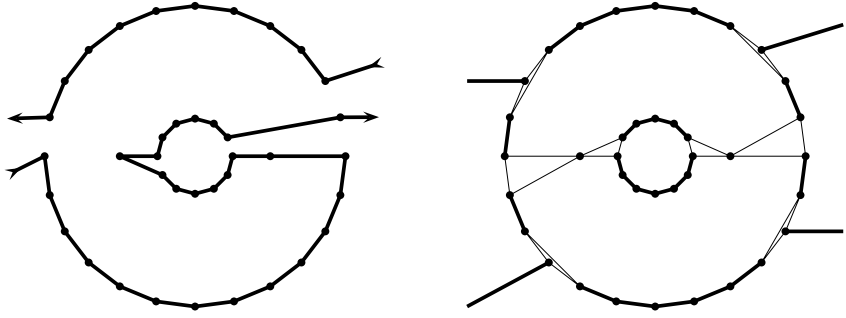

We now modify the portion $\cL^\circ$ so that the result is still a solution to the Held-Karp LP, as follows.  We still let $y_1,y_2$ and $z_1,z_2$ be pairs of points on the outer circle which are closest to the gap vertices $g_1$ and $g_2$, respectively, and similarly let $a_1,a_2$ and $b_1,b_2$ be points of the inner circle which are closest to the first and second gap vertex, respectively.  We join all pairs among each of the triples $y_1,y_2,g_1$, $a_1,a_2,g_1$, $z_1,z_2,g_2$, $b_1,b_2,g_2$ by edges of weight $\frac 1 2$.  Next, we let $\a^i_1,\a^i_2,\a^i_3$ be $\b^i_1,\b^i_2,\b^i_3$ be consecutive triples of points on the outer circle which are close to the points $x^i_1$ and $x^i_2$, respectively.  (All named points must be distinct)  We join all pairs among each of the triples $\a^i_1,\a^i_2,\a^i_3$ and $\b^i_1,\b^i_2,\b^i_3$ with edges of weight $\frac 1 2$ for $i=1,2$.  We join $\a^i_2$ and $\b^i_2$ to $x^i_1$ and $x^i_2$, respectively, for $i=1,2$, by edges of weight 1.   Finally, using edges of weight 1, we join all consecutive pairs of points on each circle which were not already joined (by edges of weight $\frac 1 2$).  The result for $S^{12}$ is shown on the right-hand side of Figure \ref{f.hkG2}.  As $k,D$ grow large and $\e$ grows small, the total cost of this is
\[
\sum_{\substack{i=1,2\\j=1,2}}\dist(x^i_j,S^k_{\e,D}) +10\pi+6+o(1),
\]
and so we have improved the length by an additive constant.

\subsubsection*{Feasibility of the solutions}
We now check that making many local modifications according to the cases above does not disturb the property that $\cL$ is a feasible instance of the Held-Karp LP.  It is immediate that all degree-weights $\sum_{j\neq i}x_{\{ij\}}$ are 2; it remains to check the condition that 
\begin{equation}\label{HKcycle}
(\forall \empt \subsetneq S\subsetneq V)\: \sum_{\{i,j\}\sbs S} x_{\{ij\}}\leq |S|-1.
\end{equation}
In other words, the total weight of edges in any proper nonempty subset is at least one less than the number of vertices.

Since the degree-weights are 2 at every vertex, we can show \eqref{HKcycle} by showing that any proper nonempty subset $S$ has the property that edges of total weight at least 2 leave the subset; i.e., that 
\begin{equation}\label{twoleave}
(\forall \empt \subsetneq S\subsetneq V)\:\: \sum_{\substack{i\in S\\j\not\in S}} x_{\{ij\}}\geq 2.
\end{equation}
If this fails, there is a cut in the graph of weight $<2$.  First we consider the possibility that the cut includes an edge of weight $\frac 1 2$.  Such edges only occur in triangles, and a minimum cut in a graph cannot contain exactly one edge of any triangle.  But this already implies that no cut of weight $<2$ in our graph can include any edges of weight $\frac 1 2$, since deleting even all the edges of any one triangle does not disconnect our graph, and since 4 such edges, or 2 such edges plus an edge of weight 1, already gives weight 2.

But now we are done: since no minimum cut of weight $<2$ can include an edge of weight $\frac 1 2$ (and all other edges have weight 1), it suffices to note that there is no single cut edge in our modified graph.  This completes the proof for the case $d=2$.

\subsubsection*{The case $d>2$}  We cannot use exactly the same point set in higher dimensions.  The only trouble with the previous argument is the lower bound on the tour length.  In particular, if an edge enters a set $S^k$ at a sharp angle to the 2-D plane containing $S^k$, then it may join to an inner circle vertex or gap vertex at negligible (or zero) extra cost over entering at the nearest possible point, and this is not accounted for in \eqref{case1cost} or \eqref{case2cost}.  Our goal now is to create a set out of many copies of $S^k$, in a way that allows us to be certain that some copy of $S^k$ must be incident only with edges nearly parallel to its containing plane, in any optimal tour.

We will let $S_{\e,D,R}^{k,\ell}$ be an $(\e,D)$ copy of a certain elaborate set $S_R^{k,\ell}$.  The set $S_R^{k,\ell}$ consists of $\ell$ copies of $S^k$.  Each copy of $S^k$ lies in a 2-dimensional hyperplane normal to some unit vector $v$, and we orient the copies of $S^k$ so that these hyperplanes are tangent to the $(d-1)$-sphere $S_R$ of radius $R$.  Moreover, we ensure that the centers $c_1,\dots,c_\ell$ lie on $S_R$, and are roughly ``evenly spaced'' in the sense that as $\ell$ grows large, we can ensure that for any given unit vector $x\in S_1$, we have that $x$ lies arbitrarily close to $\frac 1 R c_i$ for one of the $c_i$'s.

We now appeal to the following fact:
\begin{lemma}
If $X$ is a subset of the unit $(d-1)$-sphere $S_1$ in $\Re^d$ of cardinality $n$, then $\tsp(X)=o(n)$.
\end{lemma}
\begin{proof}
We can rescale the sphere and embed it in the hypercube $[0,1]^d$, and then appeal to the fact (see e.g., \cite{Few,S-Worst}) that the \emph{worst-case} length of a tour through $n$ points in $[0,1]^d$ is $\leq C_d n^{\frac {d-1}{d}}$ for some constant $C_d$.  (It is also easy to obtain a bound $C'_dn^{\frac {d-2}{d-1}}$ for the sphere, but this is unnecessary for us.)
\end{proof}
In particular, this lemma implies that if we take $R\gg \ell\gg k$, then typical edges joining pairs of instances of (approximate) $S^k$'s in the set $S_{\e,D,R}^{k,\ell}$ are of length $o(R)$ (where the asymptotics are as $R\to \infty$).  Furthermore, Observation \ref{nothree} implies that the tour enters $S_{\e,D,R}^{k,\ell}$ at most twice, and also that it enters each $S^k$ at most twice.  In particular, for sufficient choices $R\gg \ell\gg k$, there will necessarily be at least one set $X$, which is an approximate $S^k$ in the set $S_{\e,D,R}^{k,\ell}$ such that
\begin{enumerate}
\item All edges entering/leaving $X$ go to other points of $S_{\e,D,R}^{k,\ell}$, and
\item All such edges have length $o(R)$.
\end{enumerate}
In particular, the edges entering and leaving $X$ will be nearly parallel to the 2-dimensional hyperplane in which $X$ (approximately) lies.  This is because 
\begin{enumerate}[(a)]
\item  the portion of the sphere within $o(R)$ of $X$ will lie within a small angular distance of $X$, and moreover,
\item the portion of the $\e$-neighborhood of the sphere within distance $o(R)$ of $X$ but also at distance at least some large constant from $X$ is also within a small angular distance of $X$.  (Observe that points in $S_{\e,D,R}^{k,\ell}\setminus X$ are a large constant distance from $X$ by our choice of $R\gg \ell$.)
\end{enumerate}

But this is now sufficient to ensure that we can modify the tour in $X$ (according to Case 1 or 2 from above) to obtain a constant additive improvement.\qed

\section{The 2-factor limit}\label{s.lim}
Here we prove Theorem \ref{t.2flimit}: $\lim_{g\to \infty} \btfg^d=\btsp^d$.
We will continue to work with $\cY_n$ as above. We divide $[0,t]^d$ into $m^d=n/L^d$ subcubes $S_\a,\a\in [m]^d$ of sidelength $L$, for some sufficiently large constant $L>0$.

 With each cube $S_\a$ we associate the $2^d$ quadrants $Q_{\a,j},j=1,2,\ldots,2^d$, whose origin is the center $s_\a$ of $S_\a$. 
We call the quadrant $Q_{a,j}$ \emph{trivial} if the quadrant intersects $[0,t]^d$ in a unit cube (in which case $S_\a$ is one of the $2^d$ corner cubes in the decomposition). 
Then for a non-negative integer $r$, we let $Q_{\a,j,r}$ denote the cubes in $Q_{\a,j}$ whose centers are at distance at most $rL$ from $s_\a$; for convenience, we call $Q_{\a,j,r}$ trivial (resp. nontrivial) whenever $Q_{\a,j}$ is, regardless of the choice of $r$.

If $Q_{\a,j,r}\subseteq [0,t]^d$ is nontrivial and $Y_{\a,j,r}$ is the number of points of $\cY_n$ that are in $Q_{\a,j,r}$ then $Y_{\a,j,r}$ is a binomial random variable with mean 
$$\a_drL\leq \E Y_{\a,j,r}\leq \b_d(rL)^d$$ 
for some constants $\a_d,\b_d>0$. Note that, away from the boundary cubes of the decomposition of $[0,t]^d$ we can use $(rL)^d$ in place of $rL$ for the lower bound, but in the worst-case, we have to reduce the exponent. We can therefore write 
\beq{xx1}
\Pr(Y_{\a,j,r}=0)\leq e^{-\g_drL}
\eeq
for some $\g_d>0$. 

Next let $\n_r$ denote the number of subcubes $S_\a$ for which there exists $j,r$ such that $Q_{a,j}$ is nontrivial, $Q_{\a,j,r}\subseteq [0,t]^d$, and $Y_{\a,j,r}=0$. Then 
\beq{den1}
\E \n_r\leq n2^de^{-\g_drL},
\eeq 
for some $\g_d>0$. We deduce from the above that 
\beq{xx2}
\n_r=0\text{ for }r\ge r_0= \frac{2}{L}(\g_d^{-1}(\log n+d\log 2).
\eeq
Now $\n_r$ is determined by $n$ independent choices for the points in $\cY_n$. Changing one point changes $\n_r$ by at most $\d_dr^d$ for some $\d_d>0$. Applying McDiarmid's inequality we see that if $t>0$ and $r<r_0$ then
\beq{den2}
\Pr(\n_r\geq \E \n_r+t)\leq \exp\set{-\frac{t^2}{n\d_d^2r^{2d}}}.
\eeq
If $\E\n_r\geq n^{2/3}$ then \eqref{den1} implies that $r\leq \frac{\log n+3d\log 2}{3\g_dL}\leq \frac{1}{10d}\log n$ for $L$ sufficiently large. It follows from \eqref{den2} with $t=\E\n_r$ that 
\beq{xx3}
\n_r\leq 2\E \n_r\text{ a.s. if }\E \n_r\geq n^{2/3}.
\eeq
When $\E Z\leq n^{2/3}$ we use \eqref{mcd1} with $\a=n^{3/4}/(\E \n_r)$ and $c=\d_dr^{d}=\log^{O(1)}n$ to obtain
\beq{xx4}
\n_r\leq n^{3/4}\text{ a.s. if }\E \n_r\leq n^{2/3}.
\eeq

Now suppose that $C_1,C_2,\ldots,C_M$ are the cycles of a minimum cost 2-factor, where $|C_i|\geq g$ for $i=1,2,\ldots,M$. Suppose first there exist $i,j$ such that there exist $S_p\ni x\in C_i$ and $S_q\ni y\in C_j$ such that $||s_p-s_q||\leq L^2$. Suppose that $(x,x')$ is an edge of $C_i$ and that $(y,y')$ is an edge of $C_j$. Then $||x-y||\leq (L+2d^{1/2})L$ and $||x'-y'||\leq ||x'-x||+||x-y||+||y'-y||$. It follows that if we delete the edges $(x,x'),(y,y')$ from $C_i,C_j$ and add the edges $(x,y),(x',y')$ then we create a single cycle out of the vertices of $C_i\cup C_j$ at a cost of at most $2||x'-y||$. By repeating this where possible, we obtain a new set of cycles $C_1',C_2',\ldots,C_{M'}'$ such that for two distinct cycles $C_i',C_j'$ the set of subcubes visited by $C_i'$ have centers that are distance at least $L^2$ from the centers of the set of subcubes visited by $C_j'$. Furthermore, the increase in cost associated with this merging is at most
\beq{bd1}
\frac{2(L+2d^{1/2})L}{g}n.
\eeq

We continue merging cycles. For $r=L+1\ldots,,r_0$ we try to merge cycles $C,C'$ for which
there is a subcube $S_i$ containing a point of $C$ whose center is within $rL$ of the center of a subcube that contains a point of $C'$. The cost of making these merges can be bounded by
\beq{bd2}
2\sum_{r=L+1}^{r_0}n2^dr^de^{-\g_drL}+n^{3/4}r_0r^d\leq 3n(2L)^de^{-\g_dL^2}
\eeq
for $L$ sufficiently large.

This is because, when we merge two cycles via subcubes at distance $rL$ we are using one of at most $\n_r$ subcubes. Further, for each such subcube there are at most $r^d$ other subcubes at distance $r$.

We argue next that after all of these merges, there can be only one cycle. Suppose that there are two cycles $C,C'$ and let $x\in C,x'\in C'$ be as close as possible. Suppose that $x\in S_a$ and that $x'\in S_b$ where $||s_a-s_b||>r_0$. If this happens then we can find a $Q_{a,j,r_0}$ or a $Q_{b,j,r_0}$ that is empty, contradiction.

It follows from this and \eqref{bd1}, \eqref{bd2} that with $L=g^{2/3}$, that after scaling to $[0,1]^d$ we find that for $g$ sufficiently large,
$$\btsp^d\leq \btfg^d+g^{-1/4}+e^{-\g_dg^{2/3}/2}$$
and this completes the proof of Theorem \ref{t.girthseq}.
\qed
\section{Branch and Bound Algorithms}\label{B&B}
In this section we prove Theorem \ref{bbth}.  Branch-and-bound is a pruning process, which can be used to search for an optimum TSP tour.  Branch-and-bound as we consider here depends on 3 choices:
\begin{enumerate}
\item A choice of heuristic to find (not always optimal) TSP tours;
\item A choice of lower bound for the TSP (such as the 2-factor, HK bound, etc.);
\item A branching strategy (giving a \emph{branch-and-bound} tree).
\end{enumerate}

As an example, we will consider the case where we use some heuristic for the TSP, the 2-factor as a lower bound, and use a branching strategy based on the 2-factor as well.

Given our point-set $\xdn$, we begin by letting $B$ be the value of the tour found by our TSP heuristic. We let $b_{x}$ be the length of the shortest 2-factor in $\xdn$.  Here $x$ represents the root of the branch-and-bound tree, which we will construct iteratively; $\La_x$ is the set of all TSP tours in $\xdn$.   Unless $b_{x}\geq B$, we do not know that $B$ is an optimal tour, so we \emph{branch} in the following way: we choose some cycle $C$ in the 2-factor we have found, and, for the edges $e_1,\dots,e_k$ of $C$, generate $k$ children $x_1,\dots,x_k$ of $x$, letting 
\[
I_{x_1}=\empt, I_{x_2}=\{e_1\}, I_{x_3}=\{e_1,e_2\},\dots
\]
\[
O_{x_1}=\{e_1\},O_{x_2}=\{e_2\},O_{x_3}=\{e_3\},\dots
\]
These are sets of required \emph{inclusions} and \emph{exclusions}, respectively.  In particular, for any vertex $v$ of our tree, $\La_v$ is the set of tours containing all edges in $I_v$ and avoiding all edges in $O_v$.  (For the root, we had $I_x=O_x=\empt$.) Thus, in this example, the vertex $x_2$ corresponds to the set of TSP tours which do contain $e_1$ but do not contain $e_2$.    For each $x_i$, we use our TSP heuristic to find a tour, with the additional constraints that the tour includes all edges in $I_{x_i}$ and excludes those in $O_{x_i}$.  Whenever we find a TSP tour shorter than the current value of $B$, we update $B$.  We also, for each $x_i$, let $b_{x_i}$ be the minimum-length 2-factor subject to the constraints $I_{x_i},O_{x_i}$.  For any $b_{x_i}$ for which $b_{x_i}$ is at least $B$, we know that no shorter tour than $B$ exists subject to $I_{x_i},O_{x_i}$, and the tree is \emph{pruned} below $x_i$, so that $x_i$ becomes a leaf of the pruned branch-and-bound tree.  For other vertices, we continue to branch as above, by adding further constraints to kill some other cycle of the minimum 2-factors found. 

This process terminates when the set $L$ of leaves of the pruned branch-and-bound tree satisfies $v\in L\implies b_v\geq B$; such a tree corresponds to a certificate that the best TSP tour found so far by our heuristic is indeed optimum.  

In general, as a \emph{branching strategy}, we allow any method to produce, given an input weighted graph, a rooted tree (the branch-and-bound tree) labeled with sets $I_v,O_v$ of edges from the graph such that:
\begin{enumerate}
\item When $v$ is a child of $u$, $I_v\supseteq I_u$ and $O_v\supseteq O_u$.
\item If the children of $u$ are $v_1,\dots,v_k$, then we have $\La_{u}=\bigcup_{i=1}^k \La_{v_i}.$
\item \label{enum.leaves} The leaves of the (unpruned) branch-and-bound tree satisfy $|\La_v|=1$.
\end{enumerate}

Following any such branching strategy and pruning when $b_v\geq B$ will eventually lead to a proof that an optimum tour has been found (assuming a reasonable TSP heuristic and lower bound), since, in the worst case, the heuristic and lower bound will match on the leaf $v$ for which $\La_v$ contains just the optimum tour.  For branch-and-bound to be efficient, we would hope that all but polynomially many vertices of the branch-and-bound tree can be pruned because of inequalities $b_v\geq B$.

We can restate Theorem \ref{bbth} more precisely as follows
\begin{theoremR}{bbth}
For any TSP heuristic, any branching strategy, and a lower bound heuristic which is $\,\tf_g$ or $\hk$, the pruned branch-and-bound tree will have $e^{\Om(n/\log^6 n)}$ leaves a.s.
\end{theoremR} 

In particular, in our proof of Theorem \ref{bbth}, we will make the most optimistic assumption regarding the TSP heuristic: we will simply assume that it always returns an optimum tour ($B$ will always be the true value of the minimum TSP).  Theorem \ref{bbth} asserts that even in this case, there can be no polynomially-sized branch-and-bound tree which certifies optimality w.h.p.

One natural strategy to prove that $L$ is large would be to show that each $\La_v$ $(v\in L)$ is small; then, $\La=\bigcup_{v\in L} \La_v$ would give that $L$ must be large.  To show that $\La_v$ is small, one can hope to argue that to have $\lb(\xdn|I_v,O_v)\geq \tsp(\xdn)$, either $I_v$ or $O_v$ must be large, severely restricting the number of tours in $\La_v$.  The problem is that while large $I_v$ does restrict the size of $\La_v$ considerably, having a large $O_v$ can be a rather weak restriction.

We will thus modify this basic approach by paying attention to a special set of tours $\bar \La$.  Given the point set $\xdn$, we will consider the division of $[0,1]^d$ into $s=\frac n {K\log n}$ boxes of sidelength $\left(\frac {K \log n}{n}\right)^{\frac 1 d}$ where $K$ is at least some sufficiently large constant.  $B_1,B_2,\dots,B_s$ denote these boxes, taken in some order such that consecutive terms are adjacent (i.e., sharing a $(d-1)$-dimensional face). Note that
$$|x-y|\leq \sqrt d\left(\frac {K \log n}{n}\right)^{\frac 1 d}\text{ if $x,y$ lie in the same box}.$$

We consider $\xdn=\{x_1,x_2,\dots,x_n\}$, and, for for each $2\leq j\leq s-1$, we let $x_j^1,x_j^2,x_j^3,x_j^4$ denote the four points $x_i\in \xdn\cap B_j$ of smallest index (this particular choice is arbitrary, and is just for definiteness).  We also choose points $x_1^3,x_1^4\in \xdn\cap B_1$ and $x_s^1,x_s^2\in \xdn\cap B_s$, again by simply choosing points of minimum index.  Letting 
\begin{equation}\label{e.I}
\cI=\{x_i^1,x_i^2\mid 1<i\leq s\}\cup \{x_i^3,x_i^4\mid 1\leq i<s\},
\end{equation}
the points in $\cI$ can be viewed as preselected ``interface points'' between the boxes $B_j$.  In particular, we let $\bar \La$ denote the set of TSP tours in $\xdn$ with the properties that, in the tour, 
\begin{enumerate}
\item $x_1^4$ is joined to $x_1^3$ by a path lying entirely in $B_1$;
\item for $1\leq j\leq s-1$, $x_j^3$ and $ x_{j+1}^1$ are adjacent;
\item for $2\leq j\leq s-1$, $x_j^1$ is joined to $x_{j}^3$ by a path lying entirely in $B_j$;
\item $x_s^1$ is joined to $x_s^2$ by a path lying entirely in $B_s$;
\item for $s\geq j\geq 2$, $x_j^2$ and $x_{j-1}^4$ are adjacent; and
\item for $s-1\geq j\geq 2$, $x_j^2$ is joined to $x_{j}^4$ by a path lying entirely in $B_j$.
\end{enumerate}
Note that tours in $\bar \La$ use only edges of length $\leq 2\sqrt d \left(\frac {K \log n}{n}\right)^{\frac 1 d}$.   Instead of using the full strength of the condition $\La=\bigcup_{v\in L} \La_v$, we will use the weaker (but apparently more useful) condition
\[
\bar \La=\bigcup_{v\in L}\bar \La_v
\]
where $\bar \La_v=\bar \La\cap \La_v$.  Intuitively, we are focusing our attention on a restricted set of tours (chosen such that the value of $\lb$ at relevant leaves $v\in L$ with nonempty $\bar \La_v$ will be close to its typical length), and this restricted set $\bar \La$ has the property that the allowable set of edges at each vertex is now small enough that having a large excluded set $O_v$ really will force $\bar \La_v=\bar\La\cap \La_v$ to be small.

Our proof will require us to analyze the performance of $\lb$ conditioned on the exclusions $I_v$ and exclusions $O_v$.   Thus we begin by proving that some simple operations preserve the property of being an feasible instance of the Held-Karp LP.
\begin{lemma}\label{hkextend}
Suppose we are given a feasible instance of the Held-Karp LP on a set $X$ of cardinality $n$.  Then the following operations will all result in feasible instances of the Held-Karp LP:
\begin{enumerate}
\item \label{op.1edge} Downweight all edges by a factor of $(1-\frac{1}{n})$, add a new pair of vertices $y,z$ joined by an edge of weight 1, and join each of $y,z$ to all vertices in $X$ by edges of weight $\frac{1}{n}$.
\item \label{op.indset} For $k<\tfrac n 2$, downweight all edges by a factor of $(1-\frac{k}{n})$, add $k$ new vertices $y_1,\dots,y_k$, and join each $y_i$ to all vertices in $X$ by edges of weight $\frac{2}{n}$.
\end{enumerate}
\end{lemma}
\begin{proof}
It is immediate for both operations that the (weighted) degree of every vertex is 2; we need to check that
\begin{equation}\label{S}
\forall \empt\subsetneq S\subsetneq V\quad \cS(S):=\sum_{e\sbs S} x_e\leq |S|-1.
\end{equation}
For Operation \ref{op.1edge}, observe first that if $\empt\sbsn S\sbsn X$, we have $\cS(S)\leq (1-\tfrac 1 {|X|})(|S|-1)$ and $\cS(X)=|X|-1$.  For $\empt\sbsn S\sbsn X$, we have 
\[
\cS(S\cup \{y\}) \leq  (1-\tfrac 1 {|X|})(|S|-1)+|S|\tfrac 1 {|X|}=|S|+\tfrac{1}{|X|}-1\leq |S\cup \{y\}|-1.
\]
and 
\[
\cS(S\cup \{y,z\}) \leq  (1-\tfrac 1 {|X|})(|S|-1)+|S|\tfrac 2 {|X|}+1=|S|+\tfrac{|S|+1}{|X|}\leq |S\cup \{y,z\}|-1.
\]
And $\cS(X\cup \{y\})=\cS(X\cup \{z\})=|X|$. 

Finally, for Operation \ref{op.indset}, we have that $\cS(S)\leq (1-\tfrac k {|X|})(|S|-1)$ if $\empt\sbsn S\sbsn X$, and $\cS(X)=(1-\tfrac k n)(|X|)=|X|-k$.  If $Y\sbs \{y_1,y_2,\dots,y_k\}$ and $\empt \sbsn S\sbs X$ then
\begin{multline}
\cS(S\cup Y)\leq  (1-\tfrac k {|X|})(|S|-\mathbf{1}_{S\neq X})+|S|\cdot |Y|\cdot \tfrac{2}{|X|}\\=
|S|+\frac{(2|Y|-k)|S|}{|X|}-\mathbf{1}_{S\neq X}\left(1-\frac{k}{|X|}\right)\\\leq
|S|+\left(|Y|-\mathbf{1}_{|Y|<k}\right)\left(1-\frac{\mathbf{1}_{S\neq X}}{|X|}\right)-\mathbf{1}_{S\neq X}\left(1-\frac{k}{|X|}\right),
\end{multline}
and we see that $(Y\sbsn \{y_1,\dots,y_k\}\vee S\neq X)\implies \cS(S\cup Y)\leq |S|+|Y|-1$.
\end{proof}

To know that the lower bound $\lb=\hk$ performs well at leaves $v\in L$ such that $\bar \La_v\neq \empt$, we will also want to patch several smaller solutions to the Held-Karp LP into a single global solution using the same edges which tours in $\bar \La$ use to cross the $B_i$'s:
\begin{lemma}\label{hkpatch}
Suppose we are given feasible solutions to the Held-Karp LP on disjoint sets $X_1,X_2,\dots,X_s$ of cardinalities $n_1,n_2\dots,n_s$.   Write $X_i=\{x_i^1,\dots,x_i^{n_i}\}$, and suppose that the edge $\{x_i^1,x_i^2\}$ has weight 1 for all $2\leq i\leq s$ and that the edge $\{x_i^3,x_i^4\}$ has weight 1 for all $1\leq i\leq s-1$.  Then we can patch these solutions into a feasible solution on $\bigcup_{i=1}^s X_i$ by 
\begin{enumerate}[(a)]
\item  \label{delete} Deleting all the edges  $\{x_i^1,x_i^2\}$ $(2\leq i\leq s)$ and $\{x_i^3,x_i^4\}$ $(1\leq i\leq s-1)$
\item  Joining $x_i^3$ to $x_{i+1}^1$ and $x_i^4$ to $x_{i+1}^2$ with edges of weight 1, for all $1\leq i\leq s-1$
\end{enumerate}
\end{lemma}
\begin{proof}
By induction on $s$, it suffices to handle the case $s=2$.  (For $s>2$, first patch together the $X_1,\dots,X_{s-1}$, to obtain a Held-Karp solution containing $\{x_1^1,x_1^2\}$ and $\{x_{s-1}^3,x_{s-1}^4\}$ with weight 1, then patch this with $X_s$.)

It is apparent that after the patching operation, the degree of every vertex is still 2.  Suppose now that $\empt\sbsn S\sbsn X_1\cup X_2$ and write $S_i=S\cap X_i$ for $i=1,2$. Let $0\leq \r\leq 2$ denote the number of edges from among $\{x_{1}^3,x_2^1\}$, $\{x_{1}^4,x_2^2\}$ whose endpoints are both in the set $S$.  Note that 
\[
\cS(S)=\cS(S_1)+\cS(S_2)+\r.
\]
We let $\bar \cS(S_i)$ denote the weight of the HK instance on $S_i$ before the patching operation (of course, $\cS(S_i)\leq \bar \cS(S_i)$) and consider cases:\\
\textbf{Case 1:} $\r=0$\\
We have $S_i\neq X_i$ for some  $i\in\{1,2\}$.  For at least one $i$, $S_i\neq \empt$, giving
\[
\cS(S)\leq \bar\cS(S_1)+\bar\cS(S_2)\leq |S_1|+|S_2|-1=|S|-1.
\]
\textbf{Case 2:} $\r=1$\\ 
Now $\r<2$ implies that $S_1\neq X_1$ or $S_2\neq X_2$, while $\r>0$ implies that $S_1\neq \empt$ and $S_2\neq \empt$.  If in fact both $S_1\neq X_1$ and $S_2\neq X_2$, then we have 
\[
\cS(S)\leq \bar\cS(S_1)+\bar\cS(S_2)+1\leq |S_1|-1+|S_2|-1+1=|S|-1,
\]
while if (without loss of generality) $S_1=X_1$ and $S_2\neq X_2$, then the deletion step (\ref{delete}) implies that $\cS(S_1)=\bar \cS(S_1)-1,$ and so we have
\[
\cS(S)\leq \bar\cS(S_1)-1+\bar\cS(S_2)+1\leq |S_1|-1+|S_2|-1+1=|S|-1.
\]
\textbf{Case 3:} $\r=2$\\
In this case, we have that the deletion step (\ref{delete}) implies that $\cS(S_i)\leq \bar \cS(S_i)-1$.  In particular, we have 
\[
\cS(S)\leq \bar \cS(S_1)-1+\bar \cS(S_2)-1+2=\bar \cS(S_1)+\bar \cS(S_2)
\]
and thus $\cS(S)\leq |S|-1$ if $\empt\neq S\neq X_1\cup X_2$, since we must have either $\empt\neq S_1\neq X_1$ or $\empt\neq S_2\neq X_2$ ($S_i=\empt$ is not possible since $\r>0$).
\end{proof}

Next we prove a concentration lemma, which is a simple modification of what appears in \cite{S1}.  This will allow us to argue that modest conditioning does not significantly alter the value of $\lb$ in the leaves of interest.  

Let $J\sbs \{1,\dots,n\}$ be any fixed set of indices, and let $X_J\sbs \xdn=\{x_1,x_2,\dots,x_n\}$ denote the random set $\{x_j\mid j\in J\}$.  Recalling $\cI$ from \eqref{e.I}, we let 
\[
\lb_i(J)=\lb\left((B_i\cap X_J)\stm \cI\right).
\]
In particular, $\lb_i(J)$ is the value of $\lb(X_J)$ restricted to the box $B_i$, after throwing away the (two or four) special points $x_i^j$.
\begin{lemma}\label{lemconc}
For $\wLB(J)=\sum_i{\lb_i}(J)$, where $\lb=\tf_g$ or $\hk$, and $|J|=\eta,$ we have
\begin{enumerate}[(a)]
\item $\E \wLB(J)\leq \b_{\lb}\eta^{(d-1)/d}+o(\eta^{(d-1)/d}).$
\item There is an absolute constant $c>0$ such that, 
$$\Pr(\wLB(J)\geq \E \wLB(J)+t)\leq \exp\set{-\frac{ct^2}{\eta^{(d-2)/d}\log^{2/d}\eta}}.$$
\end{enumerate}
\end{lemma}
\begin{proof}\ \\
(a) Suppose that $X_J\cap B_i$ contains $n_i$ points for $i=1,2,\ldots,s$. Then, in particular, $X_J\cap B_i\stm \cI$ is $n_i-2$ or $n_i-4$ for each $i$, and we have
$$\E \wLB(J)\leq (1+o(1))\sum_{j=1}^s\frac{\b_{\lb}n_j^{(d-1)/d}}{s^{1/d}}\leq (1+o(1))\frac{\b_{\lb}}{s^{1/d}}\times s\bfrac{\sum_{j=1}^sn_j}{s}^{(d-1)/d},$$
where we have used Jensen's inequality. 

(b) Assume first that $\lb=\tf_g$. Let $J=\{j_1,j_2,\dots,j_\eta\}$.  
Let
$$d_i=\max_{j_i,\hat j_i}|\wLB(\{j_1,j_2,\ldots,j_i,\ldots,j_\eta)-\wLB(\{j_1,j_2,\ldots,\hat{j}_i,\ldots,j_\eta\})|.$$
The Azuma-Hoeffding inequality implies that
\beq{AzH}
\Pr(\wLB\geq \E \wLB+t)\leq \exp\set{-\frac{t^2}{2\sum_{i=1}^n|d_i|^2}}.
\eeq

Now fix $j_1,j_2,\ldots,j_\eta,\hat j_i$ and suppose that $j_i\in B_k$ and $\hat j_i\in B_l$. Then let
$$\D=|\wLB(j_1,j_2,\ldots,j_i,\ldots,j_\eta)-\wLB(j_1,j_2,\ldots,\hat{j}_i,\ldots,j_\eta)|.$$
Let $F=F_1\cup\cdots\cup F_s$ where $F_i$ is the optimal 2-factor for $(X_J\cap B_j)\stm \cI$, for $i=1,2,\ldots,s$. Suppose that the neighbors of $x=x_{j_i}$ on its cycle $C$ in $F$ are $y,z$. If $|C|=g$ then we cannot simply delete $x$ and replace the path $(y,x,z)$ by $(y,z)$ as this will produce a 2-factor of girth $g-1$. So, let $a$ be the closest point in $B_k$ to $x$ that is not on $C$ and let $b$ be a neighbor of $a$ on the cycle $C'$ of $F$ that contains $a$. The first thing we do now is to delete $x$ and merge the points in $C\cup C'\setminus\set{x}$ into one cycle. We delete the edges $\set{x,y},\set{x,z},\set{a,b}$ and add the edges $\set{y,a},\set{z,b}$. The change in cost is at most $2d^{1/2}\left(\frac {K \log n}{n}\right)^{\frac 1 d}$.
The new cycle has length at least $2g-1\geq g$. After this we can insert $\hx=\hx_i$ into the cycle $D$, say, of $F_l$ that contains the point $c$ of $x_1,\ldots,x_{i-1},x_{i+1},\ldots,x_n$ closest to $\hx$. 
Thus,
\beq{change}
d_i\leq 4d^{1/2}\left(\frac {K \log n}{n}\right)^{\frac 1 d}.
\eeq
This proves the lemma for $\lb=\tf_g$.

Assume now that $\lb=\hk$. Here we can use the results of Goemans and Bertsimas \cite{GB}. 
Proposition 3 and Lemma 11 of the same paper shows that $d_i\leq 2\min_{j\neq i}|\hx_i-x_j|$ and the proof goes through as before.
\end{proof}

\bigskip
We are ready to proceed with the proof of Theorem \ref{bbth}.  We consider the sizes of $\bar \La$ and $\bar \La_v=\bar \La\,\cap \,\La_v$.  We let $\b_j=|\xdn \cap B_j|$.  Then we have that 
\[
\bar \La=(\b_1-2)!\left(\prod_{j=2}^{s-1}(\b_j-3)!\right)(\b_s-2)!.
\]
Now, given $I_v$, we let $I_v^j\sbs I_v$ denote those edges in $I_v$ whose endpoints both lie in $B_j$, and $I_v'\sbs I_v$ denote those edges in $I_v$ of the form $\{x_j^3,x_{j+1}^1\}$ or $\{x_j^2,x_{j+1}^4\}$.  Given $O_v$, we let $O_v^j$ denote the set of those edges in $O_v$ whose endpoints both lie in $B_j$ and are not equal to the edges $\{x_j^1,x_j^2\},\{x_j^3,x_j^4\},\{x_j^1,x_j^4\},$ or $\{x_j^2,x_j^3\}$. 
Observe now that $\bar \La=\empt$ unless $I_v=I_v'\cup \bigcup_{j=1}^s I_v^j$ and $O_v=\bigcup_{j=1}^s O_v^j$.
We now observe that 
\begin{equation}\label{omiv}
\bar \La_v\leq (\b_1-2-|I_v^1|)!2^{|I_v^1|}\left(\prod_{j=2}^{s-1}(\b_j-3-|I_v^j|)!2^{|I_v^j|}\right)(\b_s-2-|I_v^s|)2^{|I_v^s|}!
\end{equation}
and, letting $\d_A=1$ when $|A|\geq 1$, and 0 otherwise:
\begin{equation}\label{omov}
\bar \La_v\leq (\b_1-2-\d_{O_v^1})(\b_1-3)!\left(\prod_{j=2}^{s-1}(\b_j-3-\d_{O_v^j})(\b_j-4)!\right)(\b_s-2-\d_{O_v^s})(\b_s-3)!
\end{equation}
since, e.g., the number of ways of covering $K_{[\b]}$ with paths from $1$ to $3$ and $2$ to $4$, respectively, while avoiding an edge $e$ which is not $\{1,2\},\{3,4\},\{1,4\},$ or $\{2,3\}$ is exactly either $(\b-3)!-(\b-4)!$ or $(\b-3)!-2(b-4)!$, depending, respectively on whether or not $e$ is incident with a vertex in $\{1,2,3,4\}$.

Observe now that the Chernoff bounds give that w.h.p all $\b_j$'s satisfy 
$\b_j<2 K\log n.$  In particular, there must be at least $|O_v|\left(\frac{1}{2K\log n}\right)^2$ $j$'s such $|O_v^j|\geq 1$, so that \eqref{omov} gives
\begin{equation}\label{ovest}
\bar \La_v\leq \bar \La\cdot\left(1-\frac {1}{2K\log n}\right)^{|\bar O_v|\left(\frac{1}{2K\log n}\right)^2}\leq e^{-|O_v|/(2K\log n)^3}
\end{equation}
where $\bar O_v=\bigcup_{j=1}^s O_v^j$.  Also, for $\bar I_v=\bigcup_{j=1}^s I_v^j$, \eqref{omiv} gives (very crudely) that, say,
\begin{equation}\label{ivest}
\bar \La_v\leq \bar \La\cdot (\tfrac 1 3)^{-|I_v|}\leq \bar \La\cdot e^{-|I_v|}
\end{equation}

Now \eqref{ovest} and \eqref{ivest} establish that large $\bar I_v$ or $\bar O_v$ forces $\bar \La_v$ to be small. (Note that this part of the argument would have failed if we were working with $\La_v$'s in place of $\bar \La_v$'s.) Thus, defining $\bar L=\{v\in L \mid \bar \La_v\neq \empt\}$, we have that $\bar \La=\bigcup_{v\in \bar L} \bar \La_v$.  In particular, since $\bar \La$ is large, we can show that the set of leaves $L$ of the branch and bound tree must be large (in fact, that $\bar L\sbs L$ is large) by showing that $v\in \bar L$ implies that either $\bar I_v$ or $\bar O_v$ is large.  This is where we use separation of constants.  Indeed, we will prove:
\begin{lemma}\label{lbbound}
Let $\lb$ be $\tf_g$ or $\hk$.  We have w.h.p. that for all $v\in \bar L$, either $|\bar I_v|+|\bar O_v|>\tfrac n{\log^3 n}$, or else that
\[
\lb(\xdn|I_v,O_v)\leq \b_{\lb}n^{\frac{d-1}{d}}+C(|\bar I_v|+|\bar O_v|)\left(\tfrac{\log n}{n}\right)^{\tfrac 1 d}+o(n^{\frac {d-1}{d}})
\]
for some constant $C$.
\end{lemma}

To use the lemma to complete the proof of Theorem \ref{bbth}, we observe that $v\in \bar L$ implies that $\lb(\xdn)\geq \tsp(\xdn)$ and thus, from Lemma \ref{lbbound}, that either $|\bar I_v|+|\bar O_v|>\tfrac {n}{\log^3 n}$, or else that
\[
|\bar I_v|+|\bar O_v|\geq (\btsp-\b_\lb)\tfrac{n}{\log n}-o(\tfrac{n}{\log n})
\]
In particular, in the latter case, the separation $\btsp>\b_{\lb}$ gives that
\[
|\bar I_v|+|\bar O_v|=\Omega(\tfrac{n}{\log n})
\]
In either case, \eqref{ovest} and \eqref{ivest} will then give that 
\[
\bar\La_v\leq \bar \La e^{-\Omega(n/\log^6 n)}
\]
and thus that 
\[
\bar L\geq e^{\Omega(n/\log^6 n)},
\]
completing the proof of Theorem \ref{bbth}.\qed

\begin{proof}[Proof of Lemma \ref{lbbound}]
We have that $v\in \bar L$ implies that $I_v=\bar I_v\cup I_v'$ and $O_v=\bar O_v$.  
Now let 
\[
\cY_n^v=\cX_n\setminus (V(\bar I_v)\stm V(\bar O_v)).
\]
 Letting $J(X)$ for $X\sbs \xdn$ denote the set such that $X=\{x_j\mid j\in J(X)\}$, Lemma \ref{lemconc} gives that
\begin{multline}\label{bugg}
\Pr\left(\exists S\subseteq \cX_n,|S|\leq \frac{n}{\log^3n}:\wLB(J(\cX_n\setminus S))\geq (\b_{\lb}+\e)n^{(d-1)/d}\right)\\
\leq \binom{n}{n/\log^3n}\exp\set{-\frac{c\e^2n^{2(d-1)/d}}{4n^{(d-2)/d}\log^{2/d}}}\leq 
\exp\set{\frac{n}{\log^2n}-\frac{c\e^2n}{4\log^{2/d}n}}=o(1).
\end{multline}
Thus, taking, e.g., $\e=\tfrac{1}{\sqrt{\log n}}$ we have that w.h.p all leaves $v\in \bar L$ satisfy 
\[
\wLB(J(\cY_n^v))<\b_\lb n^{(d-1)/d}+o(n^{(d-1)/d}).
\]
Recall that an instance of $\wLB$ consists of independent instances $H_j$ of $\lb$ in each $B_j\cap \ydn\stm \cI$.  Now, for each $j$, we patch into $H_j$ the edges in $\bar I_v\cap B_j$ (using Operation \ref{op.1edge} of Lemma \ref{hkextend}, if $\lb=\hk$), the endpoints of the edges in $\bar O_v\cap B_j$ (using Operation \ref{op.indset} of Lemma \ref{hkextend}, if $\lb=\hk$), and the (one or two, depending in $j$) edges 
\begin{align*}
\{x_j^1,x_j^2\}&\mbox{ for }2\leq j\leq s\\
\{x_j^3,x_j^4\}&\mbox{ for }1\leq j\leq s-1
\end{align*}
(again using Operation \ref{op.1edge}, if $\lb=\hk$).  In the case where $\lb=\tf_g$ the patching is simply accomplished by rerouting cycles through the edges and points.  Note that, as the squares $B_j$ have diameter $\sqrt d \left(\frac{K\log n}{n}\right)^{1/d},$ the total increase in cost due to this patching is $\leq C_d(|\bar I_v|+|\bar O_v|)\left(\frac{K\log n}{n}\right)^{1/d}$ for some constant $C_d$.

Next, we patch the resulting solutions together at the points $x_i^j$ to a global instance of the $\lb$ (using Lemma \ref{hkpatch}, if $\lb=\hk$).  It is important to note that by beginning with an instance of $\wLB$ and then patching to this particular instance of $\lb$, we are guaranteed this instance includes any edges in $I_v'$.  Since there are $O(\frac{n}{\log n})$ squares of diameter $O(\sqrt d(\frac{\log n}{n})^{\tfrac 1 d})$, the total increased cost from this patching is $o(n^{(d-1)/d})$, and so we have shown that
\begin{multline}
\lb(\xdn|\bar I_v\cup I_v',O_v)<\wLB(\cY_n^v)+C_d(|\bar I_v|+|\bar O_v|)\left(\tfrac{\log n}{n}\right)^{1/d}+o(n^{(d-1)/d})\\<
\b_{\lb}n^{(d-1)/d}+C_d(|\bar I_v|+|\bar O_v|)\left(\tfrac{\log n}{n}\right)^{1/d}+o(n^{(d-1)/d}),
\end{multline}
as desired.
\end{proof}

\section{Final Remarks}
Our results lead to many natural directions of inquiry, and here we mention just a few.  Apart from simply increasing the list of separated pairs of constants, the following seems like a very good challenge:
\begin{q}
What is the relationship between $\bmst^d$, $\btf^d$, $2\bmm^d$?
\end{q}

In connection with Theorem \ref{t.2flimit}:
\begin{q}
The minimum length of covering of $\xdn$ by paths of lengths at least $k$ is a Euclidean functional; let $\b^d_{P,k}$ denote the constant in its asymptotic formula.  Is it true that $\lim\limits_{k\to \infty} \b^d_{P,k}=\btsp^d$?
\end{q}

Short of a full confirmation of Conjecture \ref{ccc}, one could warm up with some special cases:
\begin{q}
Pick an integer $k$, and then prove or disprove that distinct unlabeled trees $T$ on $k$ vertices have distinct asymptotic constants $\b_T^d$.
\end{q}

Finally, we note that as our methods for separating constants give only very small differences, we have not attempted to calculate lower bounds on, say, $\btsp^d-\bhk^d$, or optimize our techniques for this purpose, though this project could be pursued.

\end{document}